\documentclass[11 pt]{amsart}
\usepackage{amscd, amssymb,url, tikz, nicefrac, stmaryrd}
\usepackage[all]{xy}

\usepackage[margin=3cm]{geometry}

\RequirePackage{tikz}
\newtheorem{thm}{Theorem}[section] 
\newtheorem{pro}[thm]{Proposition} 
\newtheorem{lem}[thm]{Lemma} 
\newtheorem{cor}[thm]{Corollary} 

\theoremstyle{definition} 
\newtheorem{defn}[thm]{Definition} 

\theoremstyle{remark} 
\newtheorem{rem}[thm]{Remark}
\newtheorem{exa}[thm]{Example}

\newcommand{\CC}{\mathbb C}

\newcommand{\NN}{\mathbb N}

\newcommand{\QQ}{\mathbb Q}

\newcommand{\ZZ}{\mathbb Z}
\newcommand{\GG}{\mathbb G}
\newcommand{\Ocal}{\mathcal O}

\newcommand{\Kcal}{\mathcal K}
\newcommand{\Mcal}{\mathcal M}
\newcommand{\Mbar}{\ol{\Mcal}}
\newcommand{\Pical}{\mathcal Pic}
\newcommand{\Mbargl}{\Mbar_g{}^{\!\!\!\ell}}
\newcommand{\upl}{{}^{\!\!\!\ell}}

\newcommand{\ttt}{{\mathrm {tw}}}

\newcommand{\Aut}{\operatorname{Aut}}
\newcommand{\Jac}{\operatorname{Jac}}

\newcommand{\Sing}{\operatorname{Sing}}
\newcommand{\Gal}{\operatorname{Gal}}

\newcommand{\im}{\operatorname{im}}

\newcommand{\Pic}{\operatorname{Pic}}
\newcommand{\Hom}{\operatorname{Hom}}

\newcommand{\Spec}{\operatorname{Spec}}

\newcommand{\zer}{\protect\ul{0}}
\newcommand{\al}{\alpha}

\newcommand{\toto}{{\mathrm{tot}\,0}}
\newcommand{\ev}{{\mathrm{ev}}}

\newcommand{\la}{\lambda}

\providecommand{\abs}[1]{\lvert#1\rvert}

\def\pmmu{{\pmb \mu}}

\newcommand{\nero}{\operatorname{\text{{N\'er}}}}
\newcommand{\nerosta}{\operatorname{\text{\emph{N\'er}}}}

\def\ol{\overline}
\def\ul{\bf}
\def\wt{\widetilde}

\newcommand{\ner}{N\'eron\ }

\newcommand{\sta}{\mathsf}

\newcommand{\Frac}{{\operatorname{Frac}}}

\newcommand{\regu}{{\operatorname{reg}}}

\begin{document}

\title{N\'eron models of $\Pic^{\zer}$ via $\Pic^{\zer}$}

\author{Alessandro Chiodo}
%

\begin{abstract}
We provide a new description of the \ner model of the 
Jacobian  of a smooth curve $C_K$ 
with stable reduction $C_R$ on a discrete valuation ring $R$
with field of fractions $K$. Instead of the 
regular semistable model, our approach uses the regular twisted model,
a twisted curve in the sense of Abramovich and  Vistoli 
whose Picard functor contains a larger separated 
subgroup than the usual Picard functor of 
$C_R$. In this way, after extracting a suitable 
$\ell$th root 
from the uniformizer of $R$,
%
%
the pullback of the \ner model of the Jacobian represents a 
Picard functor $\Pic^{\zer,\ell}$ 
of line bundles of degree zero on all irreducible components of 
a twisted curve. Over $R$, the group scheme $\Pic^{\zer,\ell}$ 
descends to the \ner model 
yielding a new geometric interpretation of its points and 
new combinatorial interpretations of the connected components of its 
special fibre. Furthermore, by construction, $\Pic^{\zer,\ell}$ is represented by a 
universal group scheme ${\mathcal{P}ic}^{\zer,\ell}_{g}$ of line bundles of degree zero
over a smooth compactification $\overline{\Mcal}_{g}\upl$ of $\Mcal_g$ 
where all \ner models of smoothings of stable curves 
are cast together after base change. 
\end{abstract}

\maketitle
\setcounter{tocdepth}{1}
{
\tableofcontents
} 

\section{Introduction}\label{sect:intro}
The N\'eron model of the Jacobian \cite{Ne, Ra}
is fundamental in the theory of
semi-stable reduction of curves 
and plays a crucial role in the study of 
compactified Jacobians. 
Indeed the Jacobian of a smooth curve over the field of fractions $K$ 
of a  discrete valuation 
ring $R$ 
 is a proper variety $\Pic^{\zer}C_K$, 
naturally equipped with the structure of an Abelian variety. 
In broad terms, the problem of compactified Jacobians aims  at finding  
a proper model for $\Pic^{\zer}C_K$. 
In general, one cannot expect a proper model which shares 
the same group properties as $\Pic^{\zer}C_K$; however, by relaxing the 
condition of properness and concentrating on the smoothness, A.~N\'eron 
discovered during 
1961--1963 
that a canonical $R$-model exists for any abelian variety $A_K$; 
this is the  N\'eron model $\nero(A_K)$. The case of the Jacobian was
completely elucidated in the late sixties by M.~Raynaud \cite{Ra} to whom 
we also owe the translation of 
N\'eron's construction in the language of schemes. 

The geometry of $\nero(\Pic^{\zer}C_K)$ is rich and 
interesting in its own right. For instance, whenever $C_K$ admits a semi-stable reduction on 
$R$, 
the group of components of the special 
fibre is the critical group $\mathcal  K(\Gamma)$ of 
the dual graph $\Gamma$ of the special fibre of the semi-stable
reduction. The critical group is the group governing sandpile 
dynamics discovered numerous times 
in different areas \cite{BHN, Biggs, Dahr, Lo1} and probably 
making their first appearance precisely in this context,  Raynaud
\cite[Prop.~8.1.2]{Ra}. We recall the main definitions in Section \ref{sect:compgroup}.

Furthermore, $\nero(\Pic^{\zer}C_K)$
sheds new light on 
the geometry of  compactified Jacobians, \emph{e.g.} \cite{Andreatta, 
Caponero,
Esteves, MV}.
We will assume for simplicity that $R$ 
is complete  with an algebraically closed residue 
field $k$. 
There is a common feature to 
well-behaved compactifications of 
the Jacobian $\Pic^{\zer}C_K$: a 
stratification of the special fibre indicating 
that  N\'eron models should be 
regarded as building blocks of the compactified Jacobians. 
Indeed, following \cite{Caponero} or \cite{Esteves} and \cite{MV},
if we admit that the curve $C_K$ has  
stable reduction $C_R$ over $R$, 
we can provide $\Pic^{\zer}C_K$ 
with a proper model $\ol J$ of $\Pic^{\zer}C_K$ 
whose special fibre $\ol J_C$ admits a stratification 
\begin{equation}\label{eq:strat}
\ol J_C= \bigsqcup_{\substack{{{S\subseteq \Sing_C}}\\
{C^{S} \text{ connected}}}}J_{C^{S}},\end{equation}
where $C$ is the stable curve, special fibre of $C_R$,
and the union runs over all possible 
sets of nodes $S\subseteq 
\Sing_C$ whose corresponding normalization $C^{S}\to C$
is connected. 
The spaces $J_{C^S}$ have co-dimension $\#S$, 
satisfy $\ol {J_{C^S}}\supseteq J_{C^{S'}}$ for $S\subseteq S'$, and 
can be identified with special fibres of 
N\'eron models of Jacobians of curves; namely 
we have 
$$J_{C^S}\cong \nero(\Pic^{\zer} C_K^S)_k.$$
where $C_K^S$ is the generic fibre of any 
smoothing of $C^S$ along $\Spec R$ and $k=R/(\pi)$ is the residue field.
The points of the compactification can be regarded as parametrising
isomorphism classes of 
line bundles 
on nodal curves; shortly in
this introduction we will illustrate the approach of Caporaso 
to which we return in \S\ref{sect:comparisons}.

Notice that $J_{C^S}\cong \nero(\Pic^{\zer} C_K^S)_k$
 is only an identification of schemes (the 
algebraic structure of \ner models is not involved)
and the right hand side of the 
isomorphism above is, 
as a scheme, merely the disconnected union 
of $c$ copies of $\Pic^{\zer} C^S$ where $c=\#\Kcal(\Gamma_S)$ is the complexity of the 
dual graph $\Gamma_S$ of $C^S$.
However, this decomposition 
arises the following natural question: 
is there a 
universal family, 
possibly equipped with a group structure, over a 
compactification of $\mathcal M_g$  
whose fibre over each isomorphism class $[C]$ is 
isomorphic to $J_C$?

\medskip 

Caporaso's approach yields the positive answer
to the analogue question where,  instead of the Jacobian,
we start from  the stack ${\mathcal Pic}^d_g$ representing the relative 
Picard functor 
$\Pic^d$ over $\mathcal M_g$ for $\gcd(2g-2,d-g+1)=1$ and $g\ge 3$ (in particular for
 $d\neq 0$). 
Indeed for these degrees Caporaso constructs a compactification 
over $\Mbar_g$ which satisfies 
the decomposition \eqref{eq:strat} after restriction to a curve over $R$.
For the functor 
$\Pic^{\zer}$ the approach can only be used in presence of 
a distinguished marking (\emph{e.g} either locally over $R$ or globally 
over $\Mbar_{g,n>1}$
where the marking labelled by $1$ allows a 
trivialisation of $\Pic^{d}C_K$ 
as a $\Pic^{\zer}C_K$-torsor and the use of the  result 
in degree-$d$). See further discussion in Section \ref{sect:comparisons}.
%

\smallskip 

When it comes to assembling 
the special fibres of N\'eron models \emph{into a group scheme}, 
it should be noticed that Caporaso's approach 
can never admit a group structure compatible with that of the Jacobian 
over $\Mcal_{g}$. 
Within the theory of nodal curves, 
one should rather refer to the group scheme introduced by
Holmes \cite{Holmes} which is fibred over a 
stack $\wt{\mathcal M}_g$ lying over $\ol {\mathcal M}_g$
and in general not quasi-comapact over it. 
This stack, which is only locally of finite presentation, 
has the striking advantage of satisfying the 
universal \ner mapping property on a higher dimensional basis. 
The
stack $\wt{\mathcal M}_g$ 
is a sort of completion of $\mathcal M_g$ because it satisfies the 
valuative criterion of properness for traits of $\ol {\mathcal M}_g$ 
whose generic points maps to $\mathcal M_g$.
We refer to Section \ref{sect:comparisons} for further discussion.

\medskip 

In this paper, we show that we can indeed assemble    
the special fibres of the \ner models of all Jacobians once we hit upon the 
right Picard functor over the right moduli space of curves. 
We describe a universal group scheme over  
a slight variation of Deligne and Mumford's 
compactification of $\Mcal_g$: the proper moduli stack $\Mbargl$ of 
of $\ell$-stable twisted curves, \emph{i.e.} 
stack-theoretic 
nodal curves with 
stabilisers of order $\ell$ at all nodes. This compactification  
is equipped with a stack $\mathcal Pic^{\zer}_g$ 
representing the relative Picard functor $\Pic^{\zer}$. Within 
$\mathcal Pic^{\zer}_g$ 
the stack $\mathcal Pic^{\zer,\ell}_g$ is locally closed and 
represents line bundles, invariant with respect to the so-called \emph{ghost} automorphisms 
fixing all geometric points and operating nontrivially as $\zeta\cdot (x,y)=(\zeta x, y)$ 
at all nodes $\{xy=0\}$ for $\zeta\in \pmmu_\ell$ (see \eqref{eq:ghostinvariantpic}).
Below, 
$R_\ell$ is the discrete valuation ring 
extracting an $\ell$th root from the uniformizer $\pi$: we set
$R_\ell=R[\wt \pi]/(\wt \pi^\ell=\pi)$. 

\medskip
\noindent\textbf{Theorem.} 
\emph{The stack $\mathcal Pic^{\zer,\ell}_g$ is a 
group scheme over  $\Mbar{}^{\,\ell}_g$ and is  
a separated model of the universal Jacobian representing 
the relative Picard functor $\Pic^{\zer}$ over $\mathcal M_g$. }

\emph{Furthermore, assume that $\ell$ is a multiple of the exponent 
of the critical 
group $\mathcal K(\Gamma)$ of any  stable  graph $\Gamma$ of genus $g$. 
Then, for any  {trait} $\Spec R\to \Mbar_g$ transversal to 
the boundary $\Mbar_g\setminus \mathcal M_g$, there exists a 
lift $\Spec R_\ell 
\to \Mbar_g{}^{\!\!\!\ell}$ 
such that 
$\Pic^{\zer,\ell}_{R_\ell}=\Pic^{\zer,\ell}_g\otimes R_\ell$
descends to $\Pic^{\zer,\ell}_R$ on $R$ and the \ner model of 
$\Pic^{\zer,\ell}_K=\Pic^{\zer,\ell}_g\otimes K$
satisfies} 
\begin{equation}\label{eq:NerviaPic}
\nero (\Pic^{\zer,\ell}_K) =\Pic^{\zer,\ell}_R.
\end{equation}

\begin{rem}
We refer to Corollary \ref{thm:global} describing explicitly the lift $\Spec R_\ell 
\to \Mbar_g{}^{\!\!\!\ell}$.
\end{rem}

\begin{rem}\label{rem:otherell}
The statement \eqref{eq:NerviaPic} 
holds for a fixed  {trait} $\Spec R\to \Mbar_g$
if and only if $\ell$ is a multiple of the exponent of 
the critical group $\mathcal K(\Gamma_k)$, where
$\Gamma_k$ is the dual graph of the curve identified by the special 
closed point $\Spec k\to \Mbar_g$ in $\Spec R$. 
 
Clearly, the statement \eqref{eq:NerviaPic} 
holds for any  {trait} $\Spec R\to \Mbar_g$
if $\ell$ is a multiple of the complexity $c(\Gamma)$ 
of all stable graphs $\Gamma$ of genus $g$ (this follows immediately from 
$\#\mathcal K(\Gamma)=c(\Gamma)$).

In Corollary \ref{cor:anothertake},
we can relax the transversality condition and consider a trait 
whose generic point still lies in $\mathcal M_g$ and 
whose strict Henselization at the special point $\Spec k \to [C]\in 
\ol {\mathcal M}_g$
$$\Spec R \to [\mathrm {Def}(C)/\Aut(C)]$$
 maps to $\pi^{t_n}$  the 
local parameter $f$ defining the Cartier 
divisor $D=(f)\subset \mathrm {Def}(C)$ 
parametrising curves where the node $n\in C$ 
persists.
Then $\ell$ is the exponent of the critical group 
of the $\bf t$-weighted graph $\mathcal K_{\bf t}(\Gamma_C)$
where each edge is decorated with the corresponding index $t_n$, see
 Section \ref{sect:compgroup} (the possibility of extending 
this formalism in the presence of weighted edges
is alluded to in \cite[\S9, Rem.~2]{BHN}). 

These considerations led us to a closer study of critical groups. We refer to Section \ref{sect:compgroup}, where we revisit the theory of critical groups, Abel theorems, 
Kirkhhoff matrix-tree theorem, and complexity in the 
context of graph with weights assigned to each edge. Some formulae simplify 
computations of the standard invariants, see \eqref{eq:retombee} and
Theorem \ref{thm:Kirchhoff}.
\end{rem}

\bigskip

We can now step back to $\Spec R$ and
compare the classical presentation of the N\'eron model of the Jacobian $\Pic^{\zer}(C_K)$
by Raynaud to the new one arising from \eqref{eq:NerviaPic}.

First, recall that, on $R$, the functor $\Pic^{\zer}$ of
line bundles of degree zero on all irreducible components, 
provides a 
separated model, but fails to satisfy the 
universal \ner  property.
The classical solution, provided  by Raynaud in 1970 \cite{Ra},   
is to enlarge $\Pic^{\zer}$ to a non-separated 
functor $\Pic^\toto$ of line bundles of total degree $0$
and to take the quotient by the closure $\ol E_K$ of the identity section 
$E_K$ over $K$. Summarizing, if we write $k$ for the residue field 
and $C_k$ for the special fibre we have 
\begin{equation}\label{eq:Raynaud}\nero(\Pic^{\zer}_K)=\Pic^\toto(C_R/R)/\ol E_K.\end{equation}
The passage to the quotient by $\ol E_K$ clearly prevents us
from defining a modular functor over $\Mbar_{g}$ 
starting from $\Pic^\toto$ and $E_K$; indeed,
the group structure defining $E_K$ 
depends on the chosen smoothing. 
Interestingly, in 2014, Holmes \cite{Holmes} 
shows  how $\ol E_K$ fails to be flat 
when we work of the moduli space of  
universal deformations of curves. As Holmes further illustrates, 
the problem persists on 
essentially any proper modification of the base space.

In this paper, we get back to $\Pic^{\zer}$, pointing out that 
 it becomes sufficiently large to 
satisfy the \ner  property 
when we consider a reduction over $R_\ell$ given by a twisted curve, \emph{i.e.} 
when we allow line bundles on stack-theoretic nodes with stabilizer $\pmmu_\ell$ 
of a twisted curve $\sta C_k$ over $C_k$.
Then \eqref{eq:NerviaPic} may be rephrased as saying that 
the points of the special fibre of the \ner
model are all represented by line bundles of degree zero
on every irreducible components. According to \eqref{eq:Raynaud}, 
the group of components of the special fibre 
is usually phrased in terms of multi-degrees of line bundles on reducible curves and 
coincides with the determinant group of 
the lattice of integral cuts (functions on 
the vertex set). Instead, in \eqref{eq:NerviaPic}, we have 
$$\nero(\Pic^{\zer}C_K)_k=\left(\Pic^{\zer}\sta C_k\right)^{\pmmu_\ell}$$ 
and the connected 
components correspond to characters at each node; this is  naturally
related to the determinant group of the lattice of integral flows
(functions on the edge set). As a byproduct 
we get a new version of the so-called Abel theorem for graphs,
see \S\ref{sect:Abel}.

We can now 
rewrite the above decomposition of the special fibre of the compactified 
Jacobian as a union of Picard groups
$$\ol J_{C}=\bigsqcup_{\substack{{S\subseteq C^{\mathrm {sing}}}\\
{C^{S} \text{ connected}}}} \left(\Pic^{\zer} \sta C^S\right)^{\pmmu_\ell},$$
where $\sta C^S$ is the twisted curve coarsely represented by $C^S$ 
with stabilisers of order $\ell$ over 
the nodes $C^S$ and 
$\ell$ is a multiple of all complexities of the curves $C^S$ (we can take 
for instance
$\ell=c(\Gamma)!$).

\subsection*{Structure of the paper} 
After setting up things in Section \ref{sect:setup},
we prove the above statement \eqref{eq:NerviaPic} and some variants 
in Section \ref{sect:Nermod}.
Then in Section \ref{sect:compgroup} we develop the combinatorial consequences 
and in Section \ref{sect:nerglobal} we prove  the theorem stated above.

\subsection*{Acknowledgements} I am are gratuful to 
Lucia Caporaso, 
Eduardo Esteves, David Holmes, Johannes Nicaise, C\'edric Pepin, 
Michel Rayanud, Matthieu Romagny, Angelo Vistoli, 
Filippo Viviani for many useful conversations. 
A special thank goes to Andr\'e Hirschowitz, who 
first encouraged me in this research direction. 



%

\section{Curves and  Picard functors}\label{sect:setup}
\subsection{Assumptions and notations}
Unless otherwise specified, we work with schemes locally of finite type over an algebraically closed field $k$. $R$ denotes a  
complete discrete valuation ring with 
algebraically closed residue field $k$ and field of fractions $K$.
We denote by $\pi$ the uniformizer of $R$.  
For any $R$-scheme 
$T_R$, we write $T_K$ for its generic fibre and $T_k$ for its special fibre.
In general for a scheme $U\to X$ and for any $X$-scheme 
$S\to X$ we write $U_S\to S$ for the corresponding base change. 

We often use strict Henselizations in order to describe
a scheme or a 
morphism between schemes locally at a closed point:
by ``local picture of $X$ at $x$'' we mean the 
strict Henselization of $X$ at $x$. 
We systematically need to extract $\ell$th roots of unity from the 
uniformizer and from the local parameters at the branches of the 
nodes we consider. 
In these cases we will need to assume that the residual characteristic 
is prime to $\ell$ and we will write $$R_\ell=R[\wt\pi]/(\wt\pi^\ell=\pi).$$
In Theorems \ref{thm:local1}, \ref{thm:local2} and \ref{thm:global},
this will force us to assume that
 $\mathrm {char}\, k$ is prime to the 
value of $\ell$ specified in the statement. 

Whenever we work with Deligne--Mumford stacks $\sta X$ we rely on 
the existence of a coarse space $X$ and a 
morphism $\sta p_{\sta X}\colon \sta X\to X$ universal with respect 
to morphisms to algebraic spaces. This allows us to associate to any 
morphism between Deligne--Mumford stacks $\sta f\colon \sta X\to \sta Y$
a morphism between algebraic spaces $f\colon X\to Y$. 

\subsection{Curves, twisted curves and families of curves}
A curve $C_k$ over $k$ always denotes a reduced, 
 connected, proper, one-dimensional
scheme over $k$ whose only singularities are nodes. 
We refer to $h^1(C_k, \Ocal_{C_k})$ as the genus $g=g(C_k)$ of $C_k$. 

\subsubsection{Dual graph} The dual graph $\Gamma$ of $C_k$ is a 
connected nonoriented graph, possibly
containing multiple edges (edges linking the same two vertices) and
loops (edges starting and ending at the same
vertex). Consider the normalisation 
$\mathrm{nor}\colon C_k^{\mathrm{nor}} \to C_k$
and the connected components of $C_k^{\mathrm{nor}}$.
The vertex set $V$ of the dual graph of $C_k$ is the set of connected components
$C_v$ of the normalisation, 
it coincides with the set of irreducible components of $C_k$. 
The edge set is the set of the nodes of $C_k$. 
A node identifies the
connected components of ${C}_k^\mathrm {nor}$ where its preimages lie,
in this way an edge links two (possibly equal) vertices.

\subsubsection{Families of curves}
A family of curves $C_S\to S$ is a proper and flat morphism whose fibres are 
curves. We consider a family of curves $C_R$ over a \emph{trait},  
the spectrum of a discrete valuation ring $R$.  
We assume that $C_K$ is smooth.
 Then, the 
local picture at a node $n$ of the special fibre is 
$\Spec R[x,y]/(xy=\pi^{t_n})$ where
$\pi$ is a uniformizer of $R$ and $t_n$ is a positive integer. 
We refer to 
$t_n$ as the \emph{thickness} of the node. 
Note that, if $t_n=1$ at all nodes, then 
$C_R$ is regular. Geometrically, we can regard the family as a morphism 
$\Spec R\to \Mbar_g$; the condition $t_e=1\ \forall e$ is a 
condition of transversality to 
the boundary $\partial \Mbar_g=\Mbar_g\setminus \Mcal_g$.

In this way, the special fibre of $C_R\to R$ yields a 
decorated graph
equipped with a function $V\to \NN$, $v\mapsto g_v:=g(C_v)$
and a function $E\to \ZZ_{\ge 1}$, $e\mapsto t_e$, 
where $t_e$ is the thickness of the node corresponding to $e$. 
We read off 
the decorated graph the genus $g=\sum_v g_v+b(\Gamma)$, where 
$b(\Gamma)$ is the first Betti number $1-|V|+|E|$ of 
the dual graph $\Gamma$.

\subsubsection{Twisted curves} A twisted curve $\sta C_k$ (in the sense of Abramovich and Vistoli) is a Deligne--Mumford stack whose coarse scheme $C_k$ is a curve.
We only consider the case where the smooth locus is represented by a scheme (we 
do not allow non-trivial stabilisers on smooth points).
At the nodes the local picture is $[(\Spec k[x,y]/(xy))/\pmmu_r]$ with 
$\zeta\in \pmmu_r$ operating as $\zeta(x,y)=(\zeta x,\zeta^{-1}y)$. The notion 
of family and of dual graph generalises word for word for a twisted curve. 
We only consider twisted curves whose coarse space is a stable curve. 

\subsubsection{The regular {{twisted}}  model}
We consider a stable curve $C_R\to \Spec R$ with $C_K$ smooth; it 
may be regarded as a stable reduction
of $C_K$ over $R$. 
Take an \'etale neighbourhood of a node $p\in C_k$ of thickness $t$:
the local picture is $\underline{U}=\Spec R[x,y]/(xy=\pi ^t)$.
Consider the quotient
stack $[U/\pmmu_t]$ where $U=\Spec R[z,w]/(zw=\pi)$ and 
$\zeta\in \pmmu_{t}$
acts as $(z,w)\mapsto(\zeta z,\zeta^{-1}w)$
The morphism $[U/\pmmu_t]\to \underline U, x\mapsto z^t, y\mapsto w^t$ is invertible away from the origin
$$\sta p = [(z=w=0)/\pmmu_t]\longrightarrow  p=(x=y=0)\in \underline U.$$
We define a twisted curve 
by glueing $[U/\pmmu_t]$ to
$C_R\setminus \{p\}$ along the isomorphism
$[U/\pmmu_t]\setminus \{\sta p\}\to \underline U\setminus\{p\}$. We repeat this procedure at all nodes of the special fibre $C_k$ and we get $\sta C^\ttt$
over $R$.
Note that $\sta C^\ttt$ is regular (as a stack over $k$).

\begin{defn}
We refer to $\sta C^\ttt\to \Spec R$ 
as the regular {{twisted}} model  associated 
to a stable reduction $C_R$ over $R$ 
of a smooth curve $C_K$ over $K$.
\end{defn}

\begin{rem}
In this context, 
the regular {{twisted}} model plays a similar role than the 
regular minimal semistable model. In particular it is 
not stable with respect to base change; 
we have chosen the notation $\sta C^\ttt$ (and we will avoid 
a notation of the form ``$\sta C_R$'') 
precisely because the construction of this model over $R'$ after base 
change of $C_R$ via an extension
$R\subseteq R'$ is not a simple base change.

We illustrate it by observing how the 
regular {{twisted}} model behaves when we pullback all data to 
the discrete valuation ring $R_\ell$ obtained by extracting an 
$\ell$th root $\pi'$ from the uniformizer $\pi\in R$.
Then the stable curve $C_{R_\ell}$, pullback of $C_R$ to $R_\ell$, is 
a stable reduction of the curve  $C_K$ pulled back 
to the field of fraction $K_\ell=\Frac (R_\ell)$.
However, the  regular {{twisted}} model associated to the stable 
reduction $C_{R_\ell}$ of $C_{K_\ell}$ is not the 
pullback of the regular {{twisted}}
model of the stable reduction $C_R$ of $C_K$. Whereas the respective
coarse spaces are 
indeed related by a simple pullback, the stacks differ: where the pullback of 
$\sta C^\ttt$ to $R_\ell$ has a stabiliser of order $t$, the 
regular {{twisted}} model of $C_{R_\ell}$ has a stabiliser of order $\ell t$.
For $\ell=1,2,\dots$ we write $\sta C^\ttt=\sta C^\ttt(1), \sta C^\ttt(2),\dots$
for these regular {{twisted}} models and we stress that the twisted curve  
$$\sta C^\ttt(\ell)\longrightarrow \Spec R_\ell$$
depends on the parameter $\ell\ge 1$ and 
is not the pullback of $\sta C^\ttt$ from 
$R$ to $R_\ell$.
\end{rem}

\subsection{The relative Picard functors}
For any stack $\sta X$ over  $k$,
we denote by $\sta {LB}({\sta X})$ the category 
of line bundles on $\sta X$ and by 
$\Pic(\sta X)$ the 
group of isomorphism classes of line bundles on $\sta X$
$$\Pic(\sta X)=H^1(\sta X,\Ocal^\times).$$

\subsubsection{Multi-degrees}
Notice that, for $\sta X=C_k$ 
a stable curve or --- more generally --- for $\sta X=\sta C_k$, 
we can decompose $\sta{LB}({\sta X})$ 
and $\Pic(\sta X)$ into 
sub-loci with fixed degree on all irreducible components. 
In the case of $\sta C_k$ we get a multi-degree whose entries are 
rational numbers (and indeed elements of 
$\frac1n\ZZ$ where $n$ is the lowest common multiple of 
the orders of the stabilisers of the points of $\sta C_k$).

\subsubsection{Relative functors}
For a family of curves 
$\sta C\to B$ there exist relative versions of these notions. 
We consider the 
fibred category $\sta {LB}({\sta C/B})$ whose objects are 
$(S,\sta L)$ formed by 
$B$-schemes $S\to B$ paired with a line bundle $\sta L$ on $\sta C_S$. 
This is represented by an Artin stack, \cite[Lem.~2.3.1]{Lieblich}.
Furthermore, for any object $\tau=(S,\sta L)$
of $\sta{LB}({\sta X/B})$
there is an embedding $\mathbb G_m(S)\hookrightarrow \Aut_S(\tau)$ compatible 
with pullbacks. Then, by \cite[Thm.~5.1.5]{ACV}
and \cite[I.~Prop.~3.0.2, (2)]{Ro} 
there exists a stack $\sta {LB}({\sta C/B})\!\!\fatslash \mathbb{G}_m$
which coincides on each fibre with the Picard group. 
We take this as a definition for the 
relative Picard functor 
$$\Pic({\sta C/B})=\sta {LB}({\sta C/B})\!\!\fatslash \mathbb{G}_m.$$
In the case of a family of twisted curves $\sta C\to B$, 
each point of the
stack $\sta {LB}({\sta C/B})\!\!\fatslash\mathbb{G}_m$
has trivial automorphism group. Therefore $\Pic({\sta C/B})$ 
is represented by a group scheme. When we work with $C_R\to \Spec R$,
 $C_K\to \Spec K$ and $C_k\to \Spec k$ (as we usually do in this paper),
 we get three group schemes which we denote by $\Pic_R, \Pic_K$ and $\Pic_k$.
 
\begin{defn}
Within the relative Picard functor of any twisted curve $\sta X$ 
over any base scheme $B$
 we can identify the following remarkable sub-functors
by imposing the following conditions to 
the restrictions of line bundles to the  fibres of $\sta X$
\begin{align*}
\Pic^{\zer} &=\text{degree zero  on 
all irreducible components of each fibre $\sta X_b$ with $b\in B$};\\
\Pic^{\mathrm{tot}\, d} & =\text{total degree $d$ on 
all irreducible components of each fibre $\sta X_b$ with $b\in B$}.
\end{align*}
\end{defn}


\subsubsection{The separated sub-group $\Pic^{\zer}$ 
of the Picard group}
When $\sta C_R\to \Spec R$ is a twisted curve 
the 
group scheme $\Pic_{R}:=\Pic(C_R/R)$
contains the above mentioned sub-group scheme
 $\Pic^{\zer}_R:=\Pic^{\zer}(C_R/R)$ representing line bundles of degree zero on every
 irreducible component of the fibres. 
\begin{pro}\label{pro:separated}
Let $\sta C_R$ be the regular {{twisted}} model of a stable 
curve $C_R$ over $R$. Then, 
$\Pic^{\zer}({\sta C_R/R})$ is a separated group scheme. 
\end{pro} 
\begin{proof} The proof is the same, word for word, as in 
\cite[Lem.~3.4,(i)]{Caponero}. We assume that $\sta L$ is a line bundle on 
$\sta C_R$ extending 
$\Ocal$ on $C_K$ with all 
degrees equal to $0$ on all components of $\sta C_k$. 
Since $\sta C_R$ is a regular Deligne--Mumford stack 
we present it in terms of a divisor supported on the special fibre. 
This can be written as $\sum_{n\in \ZZ} nD_n$ with $\cup D_n=\sta C_k$
and two distinct $D_{n'}$ and $D_{n''}$ overlapping 
only at a set of nodes.
The claim is that for $m=\min \{n\mid D_n\neq \varnothing\}$ 
we have $D_m=\sta C_k$. This claim is enough to 
prove that $\sta L$ is a pullback from 
$\Spec R$, \emph{i.e.} it represents 
the identity section of $\Pic({\sta C_R/R})$.
Indeed, this claim follows from computing 
$\deg \sta L\mid_{D_m}\ge \sum_{n>m} D_n\cdot D_m\ge 0$. Notice that 
$\sta L$ has degree zero when restricted to $D_n$; we get $\sum_{n>m} D_n\cdot D_m=0$,
which
means $D_n=\varnothing$ all $n>m$ as desired. 
\end{proof}

\subsubsection{The ghost action on 
$\Pic^0({\sta C^\ttt(\ell)/R_\ell})$}\label{sect:ghosto}
The regular {{twisted}} model $\sta C^{\ttt}(\ell)$ 
attached to $C_{R_\ell}$ is equipped with a 
$\pmmu_\ell$-action operating on the base ring $R_\ell$
by multiplication on $\wt \pi$ 
and, equivariantly, on $\sta C^{\ttt}(\ell)$. 
We consider the local picture at 
a node  $[U/\pmmu_{\ell t}]$ with 
$U=\Spec R_\ell [z,w]/(zw=\wt \pi)$ and 
$\zeta\in \pmmu_{\ell t}$ operating as 
$\zeta\cdot (z,w,\pi')=
(\zeta z,\zeta^{-1}w,\pi')$. 
The action of 
$\eta\in\pmmu_\ell\subseteq \pmmu_{\ell t}$ 
on $[U/\pmmu_{\ell t}]$ 
is $\eta\cdot (z,w,\pi')=(\eta z, w, \eta \pi')$. 
This action is $2$-isomorphic 
to $(z, \eta w, \eta \pi')$; so, no coordinate has been privileged 
(the $2$-isomorphism is realised for instance 
by the natural transformation 
$(z,w,\pi')\mapsto (\zeta^{-t}z,\zeta^tw,\pi')$
 if $\zeta^t=\eta$).
This action restricts on the special fibre to an automorphism 
which fixes 
the coarse space of $\sta C_k$; therefore, it is usually referred to 
as a ghost automorphism. 

We consider the $\pmmu_\ell$-action on the special fibre of 
$$\Pic^{\zer}({\sta C^{\ttt}(\ell)/R_\ell})
\longrightarrow \Spec R_\ell.$$
We describe the connected components of the special fibre and among
them we identify those who are formed by fixed points. In fact, 
their union 
is the fixed set, with respect to the $\pmmu_\ell$-action 
of the special fibre. Before, we need to develop the combinatorics
of the decorated graph $\Gamma$.

\subsection{The combinatorics of the graph $\Gamma$}
We consider the dual graph, its vertex set $V$ 
and its edge set $E$.
Let $\mathbb E$ be the double cover of $E$ formed by 
oriented edges. For any $e\in \mathbb E$, 
we write $\ol e$ for the oriented 
edge obtained from $e$ by reversing its orientation.
For $e\in \mathbb E$ we write $e_+$ and $e_-$ for
its tip and its tail in $V$. 
We recall that $\Gamma$ is decorated by the genus function $v\mapsto g_v$
and by the thickness function $e\mapsto t_e$, it has genus 
$g=\sum_v g_v+b(\Gamma)$ and 
is stable if for any vertex $2g_v-2+n_v$ is positive. 

\subsubsection{The cochain differential}
For  any group $G$, 
$C^0(\Gamma;G)=\{f\colon V\to G\}$ 
denotes the group of $G$-valued $0$-cochains,
 $C^1(\Gamma,G)=\{g\colon \mathbb E\to G
\mid g(e)=-g(\ol e)\}$ the group of 
 $\QQ$-valued $1$-cochains, $\delta$ 
denotes the differential 
\begin{align*}\delta \colon C^0(\Gamma;G)\to C^1(\Gamma;G),
\qquad f\mapsto \large (e\mapsto f(e_+)-f(e_-)).
\end{align*}

\subsubsection{Pairings}
When the group  $G$ is equal to the field of rational numbers $\QQ$,
 we have the following perfect pairings which 
depend on the thickenesses 
$$\langle f_1,f_2\rangle_0 = \sum_{v\in V} f_1(v)f_2(v)
\qquad \text{and} \qquad \langle g_1,g_2\rangle_0 =\frac 12
\sum_{e\in \mathbb E} \frac1{t_e} (g_1(e)g_2(e)).$$
The adjoint of $\delta$ with respect to the above bilinear 
forms is the  differential 
\begin{align*}\partial^{\bf t} \colon C^1(\Gamma;\QQ)\to C^0(\Gamma;\QQ),\qquad 
g\mapsto \left (v\mapsto \frac{1}2 \sum_
{\substack{{e\in \mathbb E}\\ {e_+=v}}} \frac{g(e)}{t_e}\right).\end{align*}
We can make the above definition explicit by choosing an orientation 
of each  edge, \emph{i.e.} a lift $E\to \mathbb E$ 
(we systematically declare when this choice is made). 
In this way $C^1(\Gamma;\QQ)$ is simply the set of 
$\QQ$-valued functions on $E$.
Then, there is a distinguished basis of $C^0$ and of $C^1$
formed by characteristic functions $\chi_v\colon v'\to \delta_{v,v'}$
and $\chi_e\colon e'\to \delta_{e,e'}$. 
Abusing the notation we identify $v$ to $\chi_v$ and $e$ to $\chi_e$
and we get 
\begin{equation}\label{eq:deltaonQ}\delta\colon \QQ^{|V|}\to \QQ^{|E|},\quad 
\delta v= \sum_{e_+=v} e-\sum_{e_-=v} e\end{equation}
and 
\begin{equation}\label{eq:partialonQ}
\partial^{\bf t}\colon \QQ^{|E|}\to \QQ^{|V|},\quad 
\partial e=\frac1{t_e}(e_+-e_-).\end{equation}

\subsubsection{Laplacian and Jacobian of the 
graph $\Gamma$ with thicknesses}
As mentioned in \cite[\S9, Rem.~2]{BHN}, there is no difficulty in extending 
the theory of the Laplacian 
in the present setup where the pairings depend on the 
thicknesses. We consider the Laplacian  
$\partial^{\bf t}\delta$ and, following the most 
widely used notation, we refer 
to the group 
$$\mathcal K_{\bf t}(\Gamma)=\frac{\partial^{\bf t}C^1(\Gamma;\ZZ)}
{\partial^{\bf t}\delta C^0(\Gamma;\ZZ)}$$
as the \emph{critical group}
of the graph $\Gamma$ decorated with the 
thicknesses $\bf t$.
This group is related---but not isomorphic in general---to 
the group of components of the special fibre of the N\'eron model. 
We refer to Proposition \ref{cor:local2}.

\subsubsection{Ordinary differentials}
When $G$ is the ring $\ZZ$, 
the bilinear forms still make sense and 
are non-degenerate. They  
induce a perfect pairing only if 
$\bf t=\bf1$. The corresponding differential $\partial^{\bf 1}$
matches the standard homology differential $\partial$ 
from $1$-chains to 
$0$-chains via the canonical identification $C_i=C^i$. 
By tensoring with any group $A$ we recover the 
differentials$$\partial_A\colon C^1(\Gamma;A)\to C^0(\Gamma;A)\quad \text{and}\quad
\delta_A\colon C^0(\Gamma;A)\to C^1(\Gamma;A)$$
between $A$-valued chains and cochains. Here we will 
use the above structure with $A=\mathbb G_m$ and $A=\QQ/\ZZ$.

\subsection{The special fibre of $\Pic^{\zer}$}
Below, in Proposition \ref{pro:pic0decomp}, we completely describe 
the connected components of the special fibre.

\subsubsection{The differential $\partial_{\bf t}$}
The  combinatorics introduced above specialises as follows. 
By restricting $\partial^{\bf t}$ from \eqref{eq:partialonQ} to $\ZZ$ we get 
$$\partial^{\bf t}\!\mid_\ZZ\colon C^1(\Gamma;\ZZ)\to C^0(\Gamma;\QQ).$$
Notice that $\partial^{\bf t}\!\!\mid_\ZZ$ 
maps the sub-group of $\ZZ$-valued 
$1$-cochains satisfying $g(e)\in t_e\ZZ$ to $\CC^0(\Gamma;\ZZ)$. 
By taking the quotient 
$C^1(\Gamma;\oplus_e \ZZ/t_e)=C^1(\Gamma;\ZZ)/C^1(\Gamma;\oplus_e (t_e))$   we get the reduced differential 
$$\partial_{\bf t}\colon
C^1(\Gamma;\oplus_e \ZZ/t_e) \longrightarrow 
C^0(\Gamma;\QQ/\ZZ).$$
After a choice of an orientation, $\partial_{\bf t}$ is simply 
\begin{align}\label{eq:partial_on_characters} \partial_{\bf t}\colon 
\bigoplus _{e\in E} \ZZ/t_e 
\longrightarrow (\QQ/\ZZ)^{|V|},
\qquad e\mapsto \frac{1}{t_e} (e_+-e_-).
\end{align}

\subsubsection{Characters at the nodes} 
Let us return to 
the regular {{twisted}} model $\sta C^{\ttt}(\ell)$.
Every choice of an orientation of an edge of the dual graph of 
the special fibre, allows us to 
choose a distinguished branch of the corresponding node. 
The fibre of a line bundle $\sta L$ over the node can be written 
for a unique choice of $a_e\in \ZZ/t_e$ as 
the $a_e$th tensor power of the line tangent 
along this distinguished branch. 
Note that $a_e(\sta L)$ changes sign if we change the orientation of $e$. 
In this way, to any line bundle $\sta L$ on 
the special fibre, we can attach an element ${\bf a}(\sta L)$ of 
the above $C^1(\Gamma;\oplus_e \ZZ/t_e)=
\bigoplus _{e\in E} \ZZ/t_e$.
We have the following definition.
\begin{defn}\label{defn:evaluationcharacters}
Consider a twisted curve $\sta X$ 
over a discrete valuation ring $R$.
We assume that the generic fibre over $K$ is smooth, 
that $\sta X$ is regular, and that the stabilisers at the nodes 
of the special fibre 
have order $t_e$. 

We have a morphism 
\begin{align}\label{eq:ev}
{\mathrm{ev}}\colon \Pic (\sta X/S)&\to C^1(\Gamma;\oplus_e \ZZ/t_e)\\
\nonumber\sta L&\mapsto {\bf a}(\sta L),
\end{align}
For any 
 ${\bf a}\in 
 C^1(\Gamma;\oplus_e \ZZ/t_e)$ 
we write $\Pic_{\bf a}({\sta X/S})$ and 
 for the sub-loci of 
line bundles mapping to $\sta a$.
\end{defn}

\subsubsection{The component group of 
the special fibre of $\Pic^{\zer}$}
For a twisted curve $\sta X$ 
over a discrete valuation ring $R$ 
with smooth generic fibres and stabilisers of order $t_e$ on 
the nodes of the special fibres 
the 
 connected components of the special fibre of 
 $\Pic^{\zer}
 (\sta X/R)$
 are classified by the \emph{kernel of characters} $\ker \partial_{\bf t}$.
 We state this in terms of the regular {{twisted}} model 
 $\sta C^{\ttt}(\ell)$,
 which is our main focus here.
 The 
 thicknesses of $\sta C^{\ttt}(\ell)$ 
 are given by $\ell \bf t$; we write $\sta p$ for the 
 morphism to the coarse space. 
  
\begin{pro}\label{pro:pic0decomp}
We have the exact sequence 
$$0\to \Pic^{\zer}(C_k) \xrightarrow {\ \sta p^* \ }  \Pic^{\zer}({\sta C^{\ttt}(\ell)/R_\ell})_k 
\xrightarrow {\ \ev \ }   
\ker \partial_{\ell\bf t}\to 0,$$
where $\Pic^{\zer}(C_k)=\Pic^0(C_k)$ is connected. 
\end{pro}
\begin{rem}\label{rem:pic0decomp}
The above statement may be regarded as a decomposition of the 
special fibre of the relative Picard functor 
$\Pic^{\zer}({\sta C^{\ttt}(\ell)}/R_\ell)$ 
into non-empty connected components isomorphic to $\Pic^{\zer}_k=
\Pic^{\zer}C_k$ 
parametrized by $\ker \partial_{\ell {\bf t}}$
$$\Pic^{\zer}({\sta C^{\ttt}(\ell)}/R_\ell)_k=\bigsqcup_{{\bf a}\in  \ker \partial_{\ell {\bf t}}} \Pic^{\zer}_{\bf a}({\sta C^{\ttt}(\ell)}/R_\ell)_k.$$
\end{rem}
\noindent \emph{Proof of Proposition \ref{pro:pic0decomp}.} 
The group $\Pic^{\zer}({\sta C^{\ttt}(\ell)/R_\ell})_k$ decomposes into open and closed 
sub-loci $\Pic^{\zer}_{\bf a}({\sta C^{\ttt}(\ell)}/R_\ell)_k$ 
parametrised by characters 
${\bf a}\in C^1(\Gamma;\oplus_e \ZZ/t_e)$. 
The identification of $\Pic^{\zer}_{\bf 0}({\sta C^{\ttt}(\ell)}/R_\ell)_k$
with $\Pic^{\zer}(C_k)$ is obvious, the line bundles that can 
be obtained as pullback from the coarse space are those 
and only those whose characters are
trivial. 

We need to check that $\bf a$ lies in 
$\ker \partial_{\ell {\bf t}}$; 
at any vertex $v$ the induced $0$-chain 
$\partial_{\ell {\bf t}}({\bf a})$
should vanish. Consider a 
line bundle $\sta L$ 
of $\Pic^{\zer}({\sta C^{\ttt}(\ell)/R_\ell})_k$, 
we write $\sta L_{\sta X}$ for the induced bundle on the 
normalisation $\sta X$ of $\sta C_v$. 
The special points $\sta p_e$ of $\sta X$ are the pre-images of the nodes and 
are naturally attached to an oriented edge $e$ directed toward
the vertex $v$. By definition 
each special point contributes $a_e/t_e\in \QQ/\ZZ$ to 
$\partial_{\ell {\bf t}}({\bf a})$. 
The value of $\partial_{\ell {\bf t}}({\bf a})$ at $v$
is $\sum_{\sta p_e\in \sta X} a_e/t_e$, which coincides with 
$\deg D-\lfloor D\rfloor$ for $\sta{L}_{\sta X}=\Ocal_{\sta X}(D)$.
Here, $\lfloor D\rfloor$ is the divisor on the coarse moduli 
space $\sta p\colon \sta X\to X$ 
attached to $\sta p_*\sta{L}_{\sta X}$; clearly $\deg\lfloor D\rfloor\in \ZZ$.
The claim follows from the fact that $\deg D=0$ because $\sta L$
is in $\Pic^{\zer}({\sta C^{\ttt}(\ell)/R_\ell})_k$.\qed

\subsubsection{The kernel of characters $\ker \partial_{\bf t}$}
For any decoration $e\mapsto t_e$, we compute $\ker\partial_{\bf t}$.
For any power of prime number $p^l$, consider the 
edges $e$ whose thickness $t_e$ is a multiple of $p^l$.
This set of edges alongside with the adjacent vertices 
determine a sub-graph  $\Gamma_p^l$ of $\Gamma$. 
For any prime number,
we have 
$$\Gamma=\Gamma_p^0\supset\Gamma_p^1\supset \dots \supset\Gamma_{p}^l
\supset \Gamma_p^{l+1}\supset \dots 
\supset \Gamma_{p}^{\max_p+1}=\varnothing,$$
where $\max_p$
is the maximum $p$-adic valuation of 
the thicknesses $t_e$.
If we set $b_{p,l}=b(\Gamma_p^l)$ we have a  
sequence $b=b_{p,0}\ge 
b_{p,1}\ge \dots b_{p,\max_p+1}=0.$ We have 
\begin{equation}\label{eq:ker_t}
\ker\partial_{\bf t}\cong \bigoplus_p 
\bigoplus_{1\le l\le \max_p} (\ZZ/p^l)^{b_{p,l}-b_{p,l+1}}.\end{equation}

\subsection{The 
Galois action on the special fibre
of $\Pic^{\zer}$}\label{sect:action}
As mentioned above, the combinatorial analysis above
allows us to describe the $\pmmu_\ell$-action on the special fibre of 
$\Pic^{\zer}$.

For any choice of $\zeta \in \pmmu_\ell$,
each vector $\pmb a\in \bigoplus \ZZ/\ell t_e$
is naturally identified to a  
$\GG_m$-valued $1$-cochain $\bf a(\zeta)$. This happens because 
$a_e\in \ZZ/\ell t_e=\Hom(\pmmu_{\ell t_e},\GG_m)$
maps $\zeta$ to $\GG_m$. 
\begin{pro}
The connected component $\Pic^{\zer}_{\bf a}({\sta C^{\ttt}(\ell)/R_\ell})_k$
within 
the special fibre 
$\Pic^{\zer}({\sta C^{\ttt}(\ell)/R_\ell})_k$
is fixed by $\pmmu_\ell$ if and only if 
$\bf a(\zeta)\in \im \delta_{\GG_m}$ for all $\zeta$ in $\pmmu_\ell$.
Furthermore, 
within the remaining connected components of
$\Pic^{\zer}({\sta C^{\ttt}(\ell)}/R_\ell)_k$,
no point is fixed. 
\end{pro}
\begin{proof} This follows immediately 
from the fact that 
$\im({\delta}\colon C^0(\Gamma;\GG_m)\to C^1(\Gamma;\GG_m))$ equals 
$\ker (\tau \colon C^1(\Gamma;\GG_m)\to \Pic C_k)$ 
and from the observation that $\zeta^*\sta L=\sta L\otimes \tau(\bf a(\zeta))$.
See \cite[eq.(29)]{Chiodo-Farkas} and \cite[Prop.~2.18]{Chistab}.
\end{proof}

\section{The N\'eron model of the Jacobian}\label{sect:Nermod}
We give two presentations of the N\'eron model
in \S\ref{sect:nerRaynaud} and 
\S\ref{sect:nerdescent}. The first presentation
is over $R$ 
and is in the same spirit of Raynaud's theorem \cite{Ra}; here,
however,
we use the regular twisted model instead of 
the regular semi-stable model. Since 
the special fibre of the 
former is more often irreducible than the latter,
the presentation of $\nero(\Pic^{\zer}C_K)$ given here is already a 
worth-mentioning
improvement. 
Then, in \S\ref{sect:nerdescent}, 
we work over a base $R_\ell$ and the N\'eron model descends
to $R$ from a group $R_\ell$-scheme  representing line bundles of degree $0$ on 
every component of the regular twisted model.

\subsection{The N\'eron model via the regular twisted model}
\label{sect:nerRaynaud}
Consider 
$G={\Pic}^{\toto}(\sta C^{\ttt})$. 
Because $\sta C^{\ttt}$ is smooth, the morphism 
$G(R)\to G(K)$ is surjective; this 
is very useful in view of the criterion 
\cite[7.1/1]{BLR}, which guarantees that a \emph{finite-type} $R$-group scheme 
is the N\'eron model of its generic fibre as soon as 
$G(R^{\textrm{sh}})\to G(K^{\textrm{sh}})$ is bijective
where $\textrm{sh}$ denotes the strict Henselization.

Here, we can interchange $R$ and its strict 
Henselization $R^{\textrm sh}$ because 
all objects involved in the statement are compatible with \'etale 
base changes (of course $\sta C^{\ttt}$ 
remains regular under any such base change). 
Recall also that N\'eron 
models descend from the strict Henselization 
of $R$ to $R$ itself (\cite[6.5/3]{BLR}).

We notice, however, that the group 
${\Pic}^{\toto}(\sta C^{\ttt})$ is \emph{not} of finite type. 
On the one hand, its 
generic fibre is of 
finite type and coincides
with $\Pic^0(C_K)$. On the other hand, 
its special fibre ${\Pic}^{\toto}(\sta C^{\ttt})_k$ 
coincides with ${\Pic}^{\toto}(\sta C_k)$, where $\sta C_k$ is 
the special fibre of $\sta C^\ttt$.
The following decomposition of the  
Picard group ${\Pic}^{\toto}(\sta C_k)$ into disconnected
components shows that the special fibre is not finite 
as soon as the curve is reducible. 
Since the regular twisted model is reducible 
more often than the the usual regular semi-stable model, 
it is worthwhile to point out a simple 
proposition summarizing the above discussion. 

\begin{pro} \label{pro:irreduciblestablemodel}
The following conditions are equivalent:
\begin{enumerate}
\item $C_K$ admits an irreducible stable model;
\item $C_K$ admits an irreducible 
regular twisted model; 
\item ${\Pic}^{\toto}(\sta C^{\ttt})$ is of finite type.
\end{enumerate}
Furthermore, under the above conditions, 
we have  
$$N(\Pic_K^{\zer}(C_K))=\Pic_R^{\zer}({\sta C}^\ttt),$$
where the functor $\Pic^{\zer}$, in this case, 
coincides with $\Pic^{\toto}$ on both sides.
\qed
\end{pro}

We generalise the above proposition by removing 
the condition $(1)$ (or their equivalent versions (2) or (3)).
We pass to the quotient 
$$G:={\Pic}^{\toto}(\sta C^{\ttt})/\ol{E},$$
where $\ol E$ is the scheme-theoretic closure
within ${\Pic}^{\toto}(\sta C^{\ttt})$ of the 
zero-section of $\Pic^{\zer}$. The group 
$\ol E$ is formed by line bundles 
of the form $\Ocal(D)$ where is $D$ is supported on the 
special fibre of $\sta C^{\ttt}$.
This quotient insures
the bijectivity of 
$G(R^{\textrm{sh}})\to G(K^{\textrm{sh}})$.
It remains to determine that $G$ is of finite type. Again, this concerns
the special fibre $G_k$ 
and amounts to determine the finiteness of the 
group of connected components of $G_k$. 

Proposition \ref{pro:pictot0decomp} is a preliminary step; it 
computes the group of connected components
of ${\Pic}^{\toto}(\sta C^{\ttt}_k)$ in terms of the image of the 
homomorphism 
$$\wt \partial^{\bf t}=(\partial^{\bf t} \times q_{\bf t})
 \colon C^1(\Gamma;\ZZ)\to C^0(\Gamma;\QQ)\times C^1(\Gamma;\oplus_e \ZZ/t_e),$$
where on the second factor, 
for each $\ZZ$-valued $1$-cochain 
$g\colon e\mapsto g(e)$, 
the morphism $q_{\bf t}$ is simply the quotient by $(t_e)$.
Below, recall that $\mathrm{ev}$ is the evaluation 
homomorphism \eqref{eq:ev} and $\underline{\deg}$ is the (possibly rational)
multi-degree.
\begin{pro}\label{pro:pictot0decomp}
The image of $({\underline{\deg}},\, \ev)\colon  {\Pic}^{\toto}(\sta C^{\ttt}_k)\to 
C^0(\Gamma;\QQ)\times C^1(\Gamma;\oplus_e \ZZ/t_e)$
lies in $\wt\partial^{\bf t} C^1(\Gamma,\ZZ)$ and 
we have the exact sequence 
$$0\to \Pic^{\zer}(C_k) \xrightarrow{\ {\sta p}^*\ } {\Pic}^{\toto}(\sta C^{\ttt}_k) \xrightarrow{\ (\underline{\deg},\, \ev) \ }  
\wt\partial^{\bf t} C^1(\Gamma,\ZZ)\to 0.$$
\end{pro}
\begin{proof}
Let us write $\sta C:= \sta C^{\ttt}_k$,  
$C:=C_k$, and $\sta p\colon \sta C
\to C$. We only need to show that $(\underline{\deg},\, \ev)$ maps to  
 $\wt\partial^{\bf t} C^1(\Gamma,\ZZ)$ and onto it.
We do so, by providing a geometrical interpretation 
to $(\underline{\deg},\,\ev)$. 

The exact sequence 
\begin{equation}\label{eq:restrictioncharacters}
0\to {\Pic}^{\toto}(C) \to 
{\Pic}^{\toto}(\sta C) 
\xrightarrow{\mathrm{ev}} C^1(\Gamma, \oplus_e \ZZ/t_e)\to 0\end{equation}
(which may be regarded as the 
long exact sequence attached to $1\to \GG_m\to 
R\sta p_*\GG_m\to R\sta p_*\GG_m/\GG_m$, 
see for instance \cite{Chistab}), 
allows us to express every line bundle $\sta L\to \sta C$
of total degree $0$ as 
$$\sta p^*\sta M_{C}
\otimes \bigotimes_{\substack{{e\in E}\\{0\le a_e< t_e}}} 
\sta M_e^{\otimes a_e},$$ 
where $\sta M_{C}$ is a line bundle of total degree $0$ 
on $C$ and 
$\sta M_e$ is the gluing of the structure sheaf $\Ocal$ on the 
regular locus $\sta C^\regu=C^\regu$ 
and the $\pmmu_{t_e}$-linearized 
line bundle $\Ocal$ on $\{zw=0\}$ given by the three characters 
$1,t_e-1,1\in \ZZ/t_e$ acting on $z,w$, and the fibre parameter
$\lambda$, respectively. 
Note that the restriction of
$\sta p_*\sta M_e$ to 
the $z^{t_e}$-branch without the origin $\Spec k[z^{t_e},z^{-t_e}]$
is the structure sheaf $\Ocal$ via the identification of 
the  parameter along the fibre  
with the $\pmmu_{t_e}$-invariant parameter $\la z^{-1}$.
Similarly  the restriction of
$\pi_*\sta M_e$ to 
the $w^{t_e}$-branch without the origin $\Spec k[w^{t_e},w^{-t_e}]$
is the structure sheaf $\Ocal$ via the identification of 
the local parameter along the fibre
with the $\pmmu_{t_e}$-invariant parameter $\la w$.
Then, the gluing is given by the canonical 
identification of the structure sheaves on $\sta C^\regu=C^\regu$ 
on the $z$-branch $\Spec k[z^{t_e},z^{-t_e}]$
and on the $w$-branch $\Spec k[w^{t_e},w^{-t_e}]$.

The local analysis above shows that 
$\sta M_e$ is a line 
bundle on $\sta C$ of multi-degree 
$\partial^{\bf t}(\chi_e)$ and that
the bundles  
$\sta M_e^{\otimes t_e}$ span $\sta p^*\Pic^{\toto}(C)$ (their 
$\pmmu_{t_e}$-linearization is trivial).
This allows us to rewrite 
\begin{equation}\label{eq:piecesofL}
\sta L\cong \sta p^*\sta M_{C}
\otimes \bigotimes_{\substack{{e\in E}\\{0\le a_e<t_e}}} 
\sta M_e^{\otimes a_e}= 
\sta p^*\sta M'_{C}\otimes \bigotimes_{\substack{{e\in E}\\{m_e\in \ZZ}}} 
\sta M_e^{\otimes m_e},\qquad \text{(with $\sta M'_C\in \Pic^{\zer}(C)$),}
\end{equation}
and to define the homomorphism 
\begin{align*}\label{eq:modoutidcomp}
\tau\colon {\Pic}^{\toto}(\sta C)&\to{\Pic}^{\toto}(\sta C)/\sta p^* {\Pic}^{\zer}(C)\\
\sta L &\mapsto \bigotimes_{\substack{{e\in E}\\{m_e\in \ZZ}}} 
\sta M_e^{\otimes m_e},
\end{align*}
whenever $(m_e)_{e\in E}$ fit in \eqref{eq:piecesofL}. 
Indeed, since any other choice $m'_e$ of $m_e\in \ZZ_e$ should be compatible with 
\eqref{eq:piecesofL}, we have $m'_e-m_e\in t_e\ZZ$. Now, notice that 
any line bundle of the form 
$\bigotimes_{e} 
\sta M_e^{\otimes bt_e}$ for $b\in \ZZ$ is a pullback from $C$.

We can now identify the map $(\underline{\deg},\,\ev)$ with 
$\wt\partial^{\bf t}\circ \tau$. Indeed 
$\underline{\deg} (\bigotimes_{e} 
\sta M_e^{\otimes m_e})=\partial^{\bf t} (m_e)$ and 
$\ev (\bigotimes_{e} 
\sta M_e^{\otimes m_e})=q_{\bf t}(m_e).$  
The surjectivity follows immediately,  any element of 
the form $(\partial^{\bf t}(m_e),q_{\bf t}(m_e))$ 
is the image of $\bigotimes_{e} 
\sta M_e^{\otimes m_e}$ via $(\underline{\deg},\,\ev)$.
\end{proof}

\begin{pro}\label{cor:local2}
The group of connected components of the special fibre of
the quotient ${\Pic}^{\toto}(\sta C^{\ttt})/\ol{E},$ is finite and it is given by the \emph{group of components}
$$\Phi_{{\bf t}}(\Gamma)=\frac{\wt\partial^{\bf t} C^1(\Gamma;\ZZ)}
{\wt \partial^{\bf t} \delta C^0(\Gamma;\ZZ)}$$
which is an extension of the 
critical group $\mathcal K_{\bf t}(\Gamma)$ by the 
kernel of characters $\ker(\partial _{{\bf t}})$ 
\begin{equation}
\label{eq:allgroups}0\to \ker(\partial _{{\bf t}})\to \Phi_{\bf t}(\Gamma)\to 
\mathcal K_{\bf t}(\Gamma)\to 0.\end{equation}
\end{pro}

\begin{proof}
We need to identify 
$(\underline{\deg},\,\ev)(\ol E)\subseteq \wt\partial^{\bf t} C^1(\Gamma,\ZZ)$.
Any line bundle of the form 
of the form $\Ocal(D)$ where is $D=\sum_{v\in V}n_v \sta C_v$
where $n_v\in \ZZ$ and $\sta C_v$ is the irreducible component 
of the special fibre of ${\sta C}^{\ttt}$ corresponding to the 
vertex $v\in V$. We set 
$$\sta L_v=\Ocal(\sta C_v)_{\mid (\sta C^\ttt)_k}$$
and show
$$(\underline{\deg},\,\ev)\Ocal(\sta C_v)=\wt \partial^{\bf t}\delta(\chi_v).$$
We have 
$$(\underline{\deg},\,\ev)(\sta L_v)=
(\underline{\deg},\,\ev)(\bigotimes_{e_+=v} 
\sta M_e)=(\partial^{\wt t}, q_{\bf e}) \sum_{e_+=v} \chi_e =
\wt\partial^{\bf t}\delta(\chi_v).$$
Finally, $\partial^{\bf t} C^1(\Gamma;\ZZ)$ contains 
$\partial C^1(\Gamma;\ZZ)$
and $\partial^{\bf t} \delta C^0(\Gamma;\ZZ)$ 
contains $\partial \delta C^0(\Gamma;\ZZ)$. Furthermore they 
are both contained in the kernel of 
$C^0(\Gamma;\QQ)\to \QQ; (\sum_v n_v[v])\to \sum v)$. Therefore 
their rank equals that of $\partial  C^1(\Gamma;\ZZ)$ and
$\partial \delta C^0(\Gamma;\ZZ)$, which is $\abs{V}-1$. 
This allows us to conclude that the quotient group is finite.
\end{proof}

\begin{thm}\label{thm:local2}
Let $C_R$ be a stable curve with smooth generic fibre 
over a complete 
discrete valuation ring $R$.  
Then,  
the \ner model of the  Jacobian of $C_K$ is 
$${\nerosta}(\Pic^{\zer}_K)={\Pic}^{\toto}(\sta C^{\ttt})/\ol E,$$
where $\ol E$ is the scheme-theoretic closure
within ${\Pic}^{\toto}(\sta C^{\ttt})$ of the 
zero-section of $\Pic^{\zer}$.
\end{thm}

\begin{proof}
Write $G$ for ${\Pic}^{\toto}(\sta C^{\ttt})/\ol E$. 
The proposition above implies that $G$ is of finite type.
Then, we can apply \cite[7.1/1]{BLR} and conclude that $G$ is a N\'eron model of 
its generic fibre by \cite[7.1/1]{BLR}. Indeed 
$G(R^{\textrm sh})\to G(K^{\textrm sh})$ is bijective (see 
\cite[7.1/1]{BLR}). 
\end{proof}

\subsection{N\'eron model of $\Pic^{\zer}$ via 
$\Pic^{\zer}$}
\label{sect:nerdescent}

We introduce  the subgroup ${\Pic}^{\zer,\ell}$ 
of the group scheme $\Pic^{\zer}(\sta C^{\ttt}(\ell)/R_\ell)$. 
Then we show that it 
singles out within $\Pic^{\zer}(\sta C^{\ttt}(\ell)/R_\ell)$ the 
part of the Picard functor that 
descends on $R$ and, when $\ell$ is conveniently chosen, 
coincides over $R$ with 
the \ner model of the Jacobian of $C_K$. 

\begin{defn}\label{defn:gal-invariant}
We denote by ${\Pic}^{\zer,\ell}_{R_\ell}$ 
the complement  within $\Pic^{\zer}(\sta C^{\ttt}(\ell)/R_\ell)$
of the connected components of the special fibre 
which are not fixed by $\pmmu_\ell=\Gal(K',K)$. 
\end{defn}

\begin{rem} The restriction of the 
map $m\colon \Pic^{\zer}\times_{R_\ell} \Pic^{\zer}\to \Pic^{\zer}$
to 
$\Pic^{\zer, \ell}\times_{R_\ell} \Pic^{\zer,\ell}$ 
factors through $\Pic^{\zer,\ell}$. 
Indeed, the pre-image of the connected components of the 
special fibre which are not fixed by $\pmmu_\ell$
lie within the special fibre of $\Pic^{\zer}\times_{R_\ell} \Pic^{\zer}$
and does not intersect 
the special fibre $\Pic^{\zer,\ell}_k \times_{k} \Pic^{\zer,\ell}_k$
because the tensor product of $\pmmu_\ell$-fixed line bundles is again a 
$\pmmu_\ell$-fixed line bundle. 
Therefore 
${\Pic}^{\zer,\ell}_{R_\ell}$
is an open sub-group scheme 
of $\Pic^{\zer}(\sta C^{\ttt}(\ell)/R_\ell)$.
\end{rem}


%
%

\begin{thm}\label{thm:local1}
Let $C_R$ be a stable curve, with smooth generic fibre, 
over a complete 
discrete valuation ring $R$. 
Then,  for any  multiple $\ell$ of the exponent of 
$\mathcal K_{\bf t}(\Gamma)$, the group scheme 
$\Pic^{0,\ell}_{R_\ell}$ descends to $\Pic^{0,\ell}_R$ 
over $R$ and 
the \ner model of the Jacobian $\Pic^{\zer}(C_K)=\Pic^{\zer,\ell}_K$
satisfies   
$$\nerosta(\Pic^{\zer,\ell}_K)={\Pic}^{\zer,\ell}_R.$$
\end{thm}

\begin{proof}
Following \cite{Caponero}, over $R_\ell$, we may write 
the pull-back of $\nero(\Pic^{\zer,\ell}_K)$ as 
$$\nero(\Pic^{\zer,\ell}_K)\otimes R_\ell= 
({\Pic}^{\toto}(\sta C^{\ttt})/\ol E)\otimes R_\ell=
\frac{\bigsqcup_{(\ul d,\ul a)} \Pic^{\ul d}_{\ul a} 
(\sta C^\ttt(\ell)/R_\ell)}{\sim K},
$$
where $\sqcup$ denotes a scheme-theoretic 
disjoint union over the indices $({\ul d},{\ul a})\in 
\Phi_{\bf t}(\Gamma)$
 and ``$\sim K$'' denotes the identification along 
the isomorphic generic fibres. 

We can rewrite the above group scheme by letting the 
the indices $(\ul d,\ul a)$ vary in 
$\wt \partial^{\bf t} C^1(\Gamma;\ZZ)$
and by modding out the action of ${\ul s}
\in  C^0(\Gamma;\ZZ)$. 
Indeed any $\ZZ$-valued cochain $\ul s$ can be 
identified to the divisor 
$\sum_{v\in V} \ell s(v) \sta C_v$
in $\sta C^{\ttt}(\ell)$.
Then we have the isomorphism 
\begin{align*}
\eta_{\ul d,\ul a}\colon \Pic^{\ul d}_{\ul a}&\rightarrow
\Pic^{\ul d'}_{\ul a'}
(\sta C^\ttt(\ell)/R_\ell)\\
L&\mapsto L\otimes \Ocal\left(\textstyle{\sum_{v\in V}\ell s(v) \sta C_v}\right),\end{align*}  
where 
$\ul d'=\ul d+\partial^{\ell \bf t}\delta \ell \ul s$ and 
$\ul a'=\ul a+q_{\ell\bf t}\delta \ell \ul s$.
We write 
$$\nero(\Pic^{\zer,\ell}_K)\otimes R_\ell= 
\frac{\left[\bigsqcup_{(\ul d,\ul a)\in\wt \partial^{\bf t} C^1} 
\Pic^{\ul d}_{\ul a} 
(\sta C^\ttt(\ell)/R_\ell)\right]_{C^0(\Gamma;\ZZ)}}{\sim K}.
$$

We recall from \S\ref{sect:action} 
that any promitive $\ell$th root of unity $\zeta$
maps $\ker \partial_{\ell\bf t}$ to $C^1(\Gamma;\GG_m)$. 
We write $\ker\partial_{\ell\bf t}\cap \im\delta_{\GG_m}$
for the elements lying in the image (this is does not depend on
the chosen primitive root $\zeta$ in $\pmmu_\ell$).
Then, the group scheme $\Pic^{\zer, \ell}_{R_\ell}$ may be
written as 
$$\Pic^{\zer, \ell}_{R_\ell}= 
\frac{\bigsqcup_{\ul a} \Pic^{\ul 0}_{\ul a} 
(\sta C^\ttt(\ell)/R_\ell)}{\sim K}.
$$
where the scheme-theoretic disjoint union runs over 
$\ul a \in \ker\partial_{\ell\bf t}\cap \im\delta_{\GG_m}$.

In order to show the 
isomorphism between the above group schemes over $R_\ell$ 
we need to use 
the condition that $\ell$ is a multiple of the exponent of 
$\mathcal K_{\bf t}(\Gamma)$ and the following lemma. 
\begin{lem}\label{lem:componentgroups}
If $\ell$ is the exponent of $\mathcal K_{\bf t}(\Gamma)$, we have
an exact sequence 
$$0\to \wt\partial^{\bf t} \delta C^0(\Gamma;\ZZ)\to 
\wt\partial ^{\bf t} C^1(\Gamma;\ZZ) \xrightarrow{\iota}  
\ker\partial_{\ell\bf t}\cap \im\delta_{\GG_m}\to 0.$$
where $\iota$ maps $\wt\partial ^{\bf t} \ul c\in \wt\partial ^{\bf t} C^1(\Gamma;\ZZ)z$ 
to $\ell \ul c-\delta \ul b$, where $\ul b$ is the unique element of 
$\delta C^0(\Gamma;\ZZ)\subseteq C^1(\Gamma;\QQ)$ 
satisfying 
$\partial^{ \bf t}\ul c =\partial^{\ell \bf t}\delta \ul b$
in $C^0(\Gamma;\QQ)$.
\end{lem}
\begin{proof}
Indeed for any element  
$(\partial^{\bf t} {\ul c}, q_{\bf t}{\ul c})
\in \wt\partial^{\bf t} C^1$ 
there exists a unique  
$\delta {\ul b}\in C^1(\Gamma;\QQ)$ 
such that 
$\partial^{ \bf t}\ul c =\partial^{\ell \bf t}\delta \ul b$.
This happens because $\ell$ is a multiple of the exponent of 
$\mathcal K_{\bf t}(\Gamma)=\partial^{\bf t}C^1/\partial^{\bf t} \delta C^0$ and, therefore, 
$\ell \mathcal K_{\bf t}(\Gamma)=\partial^{\bf t}C^1/\partial^{\ell\bf t} \delta C^0$ vanishes; hence, any 
element of the form $\partial^{\bf t} \ul c$  
lies in $\partial^{\ell\bf t} \delta C^0$.
The uniqueness follows because $\partial^{\ell \bf t}$ is the adjoint of $\delta$.

For $\wt \partial^{\bf t} c=(\ul d, \ul a)$, we set  
$$\iota(\ul d, \ul a) =
\ell \ul c- \delta \ul b\in 
\ker\partial_{\ell\bf t}\cap \im \delta_{\GG_m},$$
and we notice that 
any other choice $\ul c'$ in $C^1(\Gamma;\ZZ)$ lifting 
$(\ul d, \ul a)$ via $\wt \partial^{\bf t}$
yields the same $1$-cochain 
$\ell\ul c'=\ell\ul c$
with values in $\ZZ/\ell t_e$ because the reduction modulo $t_e$ 
of $\ul c$ and $\ul c'$ coincides with $\ul a$. 
Furthermore we have $\partial^{ \bf t}\ul c'=\partial^{ \bf t}\ul c=\ul d$, so 
${\ul c}'$ yields the same element $\delta {\ul b}'=\delta {\ul b}$, uniquely 
determined by the relation $\partial^{ \bf t}\ul c' =
\partial^{\ell \bf t}\delta \ul b'$.
The uniqueness of $\delta \ul b$ also guarantees that $\iota$ is an homomorphism
of groups. 

It is now obvious that $\iota$ vanishes on 
$\wt \partial_{\bf t}\delta C^0$.
Furthermore, 
any element in $\ker\partial_{\ell\bf t}\cap \im \delta_{\GG_m}$ can be written as the sum of a $1$-cochain 
of the form $\ell \ul c$ and a $1$-cochain lying in $\im\delta$.
Here, each co-ordinate is identified up to a multiple 
$\ell t_e d_e\in \ell t_e\ZZ$; hence, 
we can represent each element of $\ker\partial_{\ell\bf t}\cap \im \delta_{\GG_m}$
as the image of 
$\ul c - (t_ed_e)_{e\in E}$ via $\iota$. 

Now $\ker \iota$ equals $\wt\partial^{\bf t} \delta C^0(\Gamma;\ZZ)$. 
Consider $\iota(\wt\partial^{\ell\bf t} \ul c)=
\ell \ul c -\delta \ul b\in \ker\partial^{\ell t}$.
If this $1$-cochain is congruent to $\ul 0$ modulo $\ell t_e$ 
on each vertex $e$, then we can assume 
$\ell {\ul c} -\delta {\ul b}=\ell (t_ed_e)_{e\in E}$
and this implies that $b$, whose image via $\delta$ lies in $C^1(\Gamma,\ell\ZZ)$,
can be chosen in $C^0(\Gamma,\ell\ZZ)$. Then, the conditions 
$\partial^{\ell \bf t}\ell {\ul c} -\delta{ \ul b}=0$
and  
$\ell {\ul c} -\delta {\ul b}=\ell (t_ed_e)_{e\in E}$
imply 
$\partial^{\bf t} \ul c=\partial^{\bf t}\delta ( \ul b/\ell)$ and 
$q_{\bf t} {\ul c}=q_{\bf t} \delta (\ul b/\ell)$, respectively; 
in other words, 
we have
$\wt \partial^{\bf t}\delta( \ul b/\ell)=\wt\partial^{\bf t} \ul c.$
\end{proof}

For each $(\ul d, \ul a)\in \wt \partial ^{\bf t}C^1(\Gamma;\ZZ)$ 
we consider $\ul m=\iota (\ul d,\ul a)$ and the isomorphism 
\begin{align*}
\pi_{\ul d ,\ul a}\colon \Pic^{\ul d}_{\ul a}&\rightarrow
\Pic^{\ul 0}_{\ul m}
(\sta C^\ttt(\ell)/R_\ell)\\
L&\mapsto L\otimes \Ocal\left(\textstyle{\sum_{v\in V}- b(v) \sta C_v}\right).
\end{align*}
 where $\ul b$ is the unique element of 
$\delta C^0(\Gamma;\ZZ)\subseteq C^1(\Gamma;\QQ)$ 
satisfying 
$\partial^{ \bf t}\ul c =\partial^{\ell \bf t}\delta \ul b$
in $C^0(\Gamma;\QQ)$.
Notice that for any ${\ul e}\in C^0(\Gamma;\ZZ)$ 
we have 
$\pi_{\ul d',\ul a'}\circ \eta_{\ul d,\ul a}=
 \pi_{\ul d,\ul a}$.
This yields the isomorphism of group schemes 
$$\nero(\Pic^{\zer,\ell}_K)\otimes R_\ell\cong \Pic^{\zer,\ell}_{R_\ell}.$$
We conclude that 
$\Pic^{\zer,\ell}_{R_\ell}$ descends to 
$\nero(\Pic^{\zer,\ell}_K)$.
\end{proof}

\section{The component group of the special fibre of $\Pic^{\zer}$}\label{sect:compgroup}
The group of components of the special fibre of 
the N\'eron model of the Jacobian is interesting in 
its own right. It can be regarded as a purely combinatorial object. It is 
the quotient of the group of $0$-cochains $\ul{b}$
with vanishing total value $\varepsilon ({\ul{b}})=\sum_V b_v$
modulo the image of the \emph{Laplacian} $\Delta=\partial \delta$
of any integer valued $0$-cochain (recall that 
$\partial$ is the adjoint of 
$\delta$). It is convenient to express $\partial, \delta$ and $\Delta$ 
in terms of matrices.
Once we fix an orientation $\partial$ is the 
incidence matrix $I=(i_{e,v})$ with $\#V$ rows and $\#E$ columns: 
the entry $i_{v,e}$ equals $\pm1$
when $v=e_\pm$ and vanishes otherwise. Then $\delta$ is given by the transposed matrix 
$I^{\mathsf T}$ and the Laplacian is given by $M=-I\,I^{\mathsf T}$.
The quotient $\ker \varepsilon/ \im(\Delta)$ is the critical group $\mathcal K(\Gamma)$. 

{Sandpile} dynamics offers a different point of view. 
Let $\chi_v$ be the characteristic function attached to a vertex: $\chi_v(w)=0$ if $w\neq v$ 
and $\chi_v(v)=1$. 
Summing a 
$0$-cochain $\partial \delta \chi_v$ to an integer valued 
$0$-cochain $\ul b$
produces the only $0$-cochain whose total number of values equals $\sum_V b_v$ 
and whose value on 
 a vertex $w\neq v$ is the same as $b_w$ plus 
 the number of oriented edges going from 
$v$ to $w$. This relates the group of components of  the special 
fibre of the N\'eron model to the dynamics of \emph{abelian sandpiles}
which studies the 
collection of indistinguishable chips distributed among the vertices of
a graph. In this model, a 
vertex $v$ topples when it sends one chip to each neighbouring vertex; this 
is described by the above operation 
$\ul b\rightsquigarrow \ul b +\partial \delta \chi_v$. 
This justifies the terminology sandpile or chip-firing. In this context 
the group 
$\partial C^1(\Gamma;\ZZ)/\partial \delta C^0(\Gamma;\ZZ)$ is referred to 
as critical group.

The group is also analogous to the algebro-gemetric notion 
\emph{Picard} group $\Pic^{\toto}$ 
of an algebraic curve. We refer to \cite{BHN} and \cite{BN} and 
where $\partial C^1(\Gamma;\ZZ)$ is regarded as a group of divisors of 
total degree $0$. Modding out $\partial \delta C^0(\Gamma;\ZZ)$ is 
often described in this context as the analogous as passing to linear 
equivalence classes. For this reasons the critical group $\mathcal K(\Gamma)$ is 
also referred to as the ``Picard group'' of the graph.

In \cite{BHN} we can also find a combinatorial definition of the analogue of the 
\emph{Jacobian} of the curve, regarded analytically as the quotient of 
the dual space of abelian differentials $H^0(C,\omega_C)^\vee$
modulo $H^1(C;\ZZ)\hookrightarrow H^0(C,\omega_C)^\vee$. Indeed, 
we can consider $1$-chains with values in $\QQ$.
Then 
the lattice $H^1(\Gamma;\ZZ)$ is the analogue of $H^1(C;\ZZ)$, whereas 
$H^1(\Gamma;\ZZ)^\#=\{{\ul x} \in H^1(\Gamma;\QQ)\mid \langle \ul x,\ul \lambda\rangle \in 
\ZZ \text{ for all $\la \in H^1(\Gamma;\ZZ)$}\}$
is the analogue of $H^0(C,\omega_C)^\vee$. 
The Jacobian is the quotient 
$$\mathcal J(\Gamma)=\frac{H^1(\Gamma;\ZZ)^\#}{H^1(\Gamma;\ZZ)}.$$

All the above groups are different presentations of the same finite Abelian group
$\mathcal K(\Gamma)$
whose cardinality equals the \emph{complexity} 
$c(\Gamma)$: the number of spanning trees within $\Gamma$. 
These are subgraphs of $\Gamma$ with the same vertex set 
$V$ as $\Gamma$ and of type tree (the Betti number $b$ equals $0$).
The complexity is efficiently computed by Kirchhoff's matrix-tree theorem
$$c(\Gamma)=(-1)^{s+t+\#V-1}\det M_{s,t}$$
where $M_{s,t}$ is obtained by suppressing the $s$th row and the $t$th line. 

The previous results allow us to go through some natural generalisations
of these notion and theorems when each edge $e$ is equipped with a 
thicknesses $t_e$. There are several interesting consequences even for the 
standard case $t_e=1$ for all $e$.

\subsection{The $\bf t$-critical group $\mathcal K_{\bf t}(\Gamma)$}
In the previous sections we encountered the following generalisations of 
the above mentioned objects. The $\bf t$-critical group of $\Gamma$ is 
$$\mathcal K_{\bf t}(\Gamma)=\frac{\partial^{\bf t}C^1(\Gamma;\ZZ)}
{\partial^{\bf t}\delta C^0(\Gamma;\ZZ)}.$$
Notice how modding our ${\partial^{\bf t}\delta C^0(\Gamma;\ZZ)}$ is 
the analogue for thick edges of the operation defining the equivalence between abelian 
sandpiles. In the present model one should think of  fractional sandpiles, where 
$v$ topples when it sends a fractional chip $1/t_e$ to the neighbouring vertex
at the opposite extreme of the edge $e$. 

\subsection{The group of components $\Phi_{\bf t}(\Gamma)$} 
The group of components is 
$$\Phi_{\bf t}(\Gamma)=\frac{\wt\partial^{\bf t}C^1(\Gamma;\ZZ)}
{\wt\partial^{\bf t}\delta C^0(\Gamma;\ZZ)}$$
and, by Theorem \ref{thm:local1} is indeed the group of components of the special fibre of the 
N\'eron model of the Jacobian of a smooth 
curve whose stable model has thicknesses 
$\bf t$. 
The group of characters is 
$\ker \partial_{\bf t}$
is the kernel of $\Phi_{\bf t}(\Gamma)\to \mathcal K_{\bf t}(\Gamma)$.

\subsection{The Jacobian $\mathcal J_{\bf t}(\Gamma)$ of a graph}
The group of characters attached to the 
thicknesses $\ell \bf t$ contain a 
copy of the group of components $\Phi_{\bf t}(\Gamma)$ for suitable choices of $\ell$
(the multiples $\ell$ of the exponent of the $\bf t$-critical group).
The subgroup isomorphic to $\Phi_{\bf t}(\Gamma)$ is the preimage of  
$\im \delta_\ell$ under the reduction modulo $t_e$ of all values at each edge $e$. 
Generalising the above notion of Jacobian of a graph $\Gamma$, 
we can present the group $\lim_\ell 
(\ker \partial_{\ell \bf t}\cap \im\delta_\ell)$
as an extension 
$$0\to \Lambda_{\bf t} \to \Lambda_{\bf t}^\#\xrightarrow{ \mod t_e\ZZ\  \ }\lim_\ell 
(\ker \partial_{\ell \bf t}\cap \im\delta_\ell)\to 0,$$
where $\Lambda$ is the lattice within $\ker\partial_{\bf t}$ whose 
values at the edge $e$ lie in $t_e\ZZ$
(as above we have $\Lambda^\#=\{{\ul x} \in \ker\partial_{\bf t}\mid \langle \ul x,\ul \lambda\rangle \in 
\ZZ \text{ for all $\la \in \Lambda$}\}$).
We set the notation
$$\mathcal J_{\bf t}(\Gamma)=\lim_\ell\left(\ker \partial_{\ell \bf t}\cap \im\delta_\ell\right)=
\frac{\Lambda_{\bf t}^\#}{\Lambda_{\bf t}}.$$

\subsection{Abel theorem for a graph with thickness $\bf t$} \label{sect:Abel}
Lemma \ref{lem:componentgroups} may be regarded as Abel's theorem for 
a graph with thicknesses $\bf t$. Indeed we have an isomorphism between the 
$\bf t$-critical group (playing the role of a sort of Picard group) and 
the Jacobian $\mathcal J_{\bf t}(\Gamma)$ of the graph:
$$\Phi_{\bf t}(\Gamma)\cong \mathcal J_{\bf t}(\Gamma).$$
Notice how this provides a somewhat different proof even 
in the known case ${\bf t}=\bf 1$. 

\subsection{Complexity for a graph with thickness $\bf t$}
For ${\bf t}={\bf 1}$, the complexity counts the number of elements of the critical group, 
which coincides with the group of components. In the general setup the 
complexity admits a natural generalisation yielding a positive integer counting 
the number of elements of the group of components. 

Let $\mathcal S$ be the set of spanning trees of $\Gamma$. Each element of $\Gamma$ 
is determined by a subset of $E$ of cardinality $\#V-1$ (this happens because 
the Betti number $b$ vanishes if and only if $1-\#V+\#E$ vanishes). 
We can define the complexity $c_{\bf t}(\Gamma)$ as the following weighted sum 
$$c_{\bf t}(\Gamma)=\sum_{S\in \mathcal S} \prod_{e\not \in S} t_e,$$
where $e \not \in s$ means that $e$ is not one of the $\#V-1$ edges of $S$ (regarded 
as a set).
Notice that any spanning tree of $\Gamma$ which does not contain 
the edges $e_1,\dots, e_r$ yields $t_{e_1}\cdot\dots\cdot t_{e_r}$ spanning trees 
of the ``blown up'' graph $\mathrm{Bl}_{\bf t}\Gamma$ obtained by subdividing the edge $e$ into $t_{e_i}$ edges. 
This is the graph of the special fibre of the regular semi-stable model. 
We have 
$c_{\bf t}(\Gamma)=c(\mathrm{Bl}_{\bf t}\Gamma)=|\Phi(\mathrm{Bl}_{\bf t}\Gamma)|.$
On the other hand, by Theorem \ref{thm:local1},
$\Phi_{\bf k}(\Gamma)$ is isomorphic to $\Phi(\mathrm{Bl}_{\bf t}\Gamma)$; hence we have 
\begin{equation}\label{eq:retombee}
c_{\bf t}(\Gamma)=|\Phi_{\bf t}( \Gamma)|=|\Phi(\mathrm{Bl}_{\bf t}\Gamma)|=
c(\mathrm{Bl}_{\bf t}\Gamma),
\end{equation}	
which can be used to simplify the computation of ordinary graphs after contraction of all 
subchains.
Now, the size of the $\bf t$-critical group $\mathcal K_{\bf t}(\Gamma)$ can 
be computed via \eqref{eq:ker_t}
\begin{equation}\label{eq:ordcrit}
|\Kcal_{\bf t}(\wt \Gamma)|=c_{\bf t}(\Gamma)\prod_p \prod_{l=0}^{\max_p} p^{-b_{p,l}}.\end{equation}
Notice that this is also a suitable index $\ell$ in Theorem \ref{thm:local2}. 

\subsection{Kirchhoff's matrix-tree theorem for graph with thickness $\bf t$}
Kirchhoff theorem matrix-tree theorem extends straightforwardly as follows.
\begin{thm}\label{thm:Kirchhoff}
We have 
$$c_{\bf t}(\Gamma)=(-1)^{a+b+\#V-1}\det M_{a,b}\prod_{e\in E} t_e,$$
where $M$ is the matrix representing the Laplacian $\partial^{\bf t}\delta$ 
as a above and 
$M_{a,b}$ is the sub-matrix obtained by suppressing the $a$th row and the $b$th column. 
\end{thm}
\begin{proof}
Let $I_{\bf t}$ be the incidence matrix of the graph with thicknesses. 
It is given by multiplying by $1/t_e$ the column corresponding to $e$.
Notice that $I_{\bf t}$ represents $\partial^{\bf t}$, $\delta$ is still represented 
by $I^{\mathsf T}$, and the Laplacian $\partial^{\bf t}\delta$ is represented by 
$M=-I_{\bf t}I^{\mathsf T}.$

It is enough to prove the identity for $a=b$.
Let 
 $ I^\star$ and $I_{\bf t}^\star$ be the sub-matrices of the 
incidence matrix $I$ and  $I_{\bf t}$ 
obtained by deleting the row $a$. 
Notice that $M_{a,a}= -I^{\star } (I^\star)^{\mathsf T}$.
We have 
$$c_{\bf t}(\Gamma)
=\sum_{S\in \mathcal S} \prod_{e\not \in S} t_e=
\sum_{S\in \mathcal S} (\det I(S))^2 \prod_{e\not \in S} t_e,$$
where $I(S)$ (resp. $I_{\bf t}(S)$) 
is the square sub-matrix of $I^\star$ determined by 
suppressing all columns attached to the vertices $e\not\in S$.
Notice that $\det I(S)= \det I_{\bf t}(S)\prod_{e\in S} t_e.$ Then we have 
\begin{multline*}c_{\bf t}(\Gamma)=
\sum_{S\in \mathcal S} \det I_{\bf t}(S)\det I(S)\prod_{e\in E} t_e =
\sum_{\substack{{S\subseteq E}\\{|S|=\#V-1}}} 
\det I_{\bf t}(S)\det I(S)  \prod_{e\in E} t_e\\=
(-1)^{\#V-1}\det M_{a,a}\prod_{e\in E} t_e,\end{multline*}
where, in the last identity,  
we used the Binet--Cauchy formula. 
\end{proof}

\section{A universal group scheme}\label{sect:nerglobal}
\subsection{The stack of $\ell$-stable curves}
Let $\Mbar_g{}^{\!\!\!\ell}$ be the 
stack of twisted curves with all stabilisers of order $\ell$
and stable moduli space. We call its objects $\ell$-stable curves 
because the condition imposing the same order on the stabilisers is indeed a stability 
condition (in the sense that a unique stable reduction exists).
Every stable dual graph $\Gamma$ determines a locally closed 
substack $\Mcal_\Gamma^\ell$, the full subcategory of $\ell$-stable curves whose dual 
graph is $\Gamma$. 

We have a stratification 
$$\Mbargl=\bigsqcup_\Gamma \Mcal_\Gamma^\ell,$$
where $\Gamma$ runs over all stable genus-$g$ graphs  and 
$\Mcal_\Gamma^\ell$ is in the closure of $\Mcal_{\Gamma'}^\ell$ if 
$\Gamma'$ can be obtained from $\Gamma$ via edge-contraction. 
Consider the analogue stratification $\Mcal_\Gamma\subset \Mbar_{g}$ within 
the Deligne--Mumford stack of stable curves; 
then the forgetful morphism $\Mcal_\Gamma^\ell\to \Mcal_\Gamma$ is 
a constant $(\pmmu_\ell)^E$-gerbe; in other words  
the automorphism group $\Aut_{C}(\sta C)$ 
classifying ghost automorphisms of $\sta C$ that fix the coarse space 
is canonically isomorphic to 
$(\pmmu_\ell)^E$ for any 
$\ell$-stable curve $\sta C$ with coarse space $C$
and the dual graph $\Gamma$ , see 
\cite[\S7]{ACV}. The  ghost automorphism group scheme, which is trivial
over  
$\Mcal_\Gamma^\ell$, contains the distinguished subgroup scheme $\Delta_\Gamma$
of diagonal automorphisms. This consists of the 
ghost automorphisms $\zeta\in \pmmu_\ell$ 
acting simultaneously at all nodes $[\{xy=0\}/\pmmu_\ell]$ 
as $\zeta\cdot(x,y)=(\zeta x,y)\equiv (x,\zeta y)$.
We have $\Delta_\Gamma=\pmmu_\ell$ over $\Mcal_\Gamma^\ell$.

\subsection{The Picard group scheme over $\Mbargl$}
Consider the separated group scheme $\Pical_g^{\zer}$ 
representing the functor $\Pic^{\zer}$ over the moduli 
stack $\Mbar_g\upl$ (the proof of the separatedness 
follows almost immediately from Lemma \ref{pro:separated}. 
On each substack $\Mcal_\Gamma^\ell$ we 
consider $\Pical_\Gamma^{\zer}$, which is a disjoint union of 
components $\Pical_{\Gamma,\ul a}^{\zer}$ parametrised by 
$\ul a\in \ker\partial_\ell$
$$\Pical_\Gamma^{\zer}=\bigsqcup_{\ul a\in \ker\partial_\ell} \Pical_{\ul a}^{\zer}.$$
 By \S\ref{sect:action} the 
the components labelled by $\al\in \ker\partial_\ell \cap \im\delta_\ell$ 
are fixed by $\Delta$ whereas in the remaining components 
no point is fixed. 
Consider 
\begin{equation}\label{eq:ghostinvariantpic}
\Pical^{\zer,\ell}_g:=\Pical_g^{\zer}\setminus \bigsqcup_{\Gamma,\ul a \not \in \im\delta_\ell} 
\Pical_{\Gamma,\ul a}^{\zer}.\end{equation}
It is a separated subgroup scheme of $\Pical_g^{\zer}$ (the invariance with respect to 
$\Delta_\Gamma$ is closed under multiplication).

We can now state Theorem \ref{thm:local2} in a global form. 
We recall that a smoothing of $C_k$ is a family $C_R$ over $R$ whose 
generic fibre is smooth and whose special fibre is isomorphic to $C_k$.
The base change via $R\subseteq R_\ell$ 
of $C_R$ is not a smoothing of the generic fibre, but is 
associated to to a unique regular stable twisted model
$\sta C^{\ttt}(\ell)\to \Spec R_\ell$. 
This is an $\ell$-stable curve and we may regard it as a morphism 
$\Spec R_\ell\to \Mbar_g{}^{\!\!\!\ell}$. 
\begin{cor}\label{thm:global} For any stable curve $C_k$ 
and for any of its smoothings $C_R$ over $R$, the \ner model of the Jacobian $\Pic_K$
of the generic fibre coincides, after pullback to $R_\ell$, with the 
base change of $\Pical^{\zer,\ell}_g\to \Mbargl$ via 
$\Spec R_\ell\to \Mbargl$, 
the morphism induced by the regular {{twisted}} model
$\Spec R_\ell\to \Mbar_g{}^{\!\!\!\ell}$ 
associated to 
$C_{R_\ell}$. Summarising, we have the following fibre diagrams
$$
\xymatrix @!0 @R=1cm @C=1.3cm {
 \Pic_K\ar[rr]\ar[dd]    && 
           N(\Pic_K)\ar[dd]  && 
                \Pic^{\zer,\ell}_{R_\ell}\ar[dd] \ar[ll]\ar[rr]&& 
                       \Pical^{\zer,\ell}_g\ar[dd] 
        \\
                             &\square&    &\square&    &\square&  \\
        \Spec K \ar[rr]  && 
           \Spec R && 
              \Spec R_\ell \ar[rr]\ar[ll]&& 
                   \Mbar_g{}^{\!\!\ell}.
        }
$$\qed
\end{cor}

\subsection{The different approaches of Caporaso and of Holmes}\label{sect:comparisons}
We discuss here two different approaches to the problem of assembling 
into a single universal family the special fibres of the \ner model of the Jacobian
$\Pic^{\zer} C_K$.

Caporaso's approach \cite{Caponero} is a subfunctor of the Picard functor. 
It is given by singling out within the Picard functor the points representing  
the so-called ``balanced'' line bundles which we recall  below. 
As a first step, we switch to the study of the 
variety $\Pic^dC_K$, where $C_K$ is the smooth curve, generic fibre of a stable curve
over $R$. We assume $g\ge 3$ and that $2g-2$ is prime to $d-g+1$, 
which rules out the case of the Jacobian $\Pic^{\zer}C_K$. However 
there will be consequences for the initial problem of \ner models of Jacbians 
in some special cases where we can rely on a trivialization of the torsor 
$\Pic^dC_K\cong \Pic^{\zer}C_K$. 

Let us illustrate first the notion of balanced multidegree. 
In rough terms a line bundle $L$ on a stable curve $C$ is 
balanced of degree $d$ essentially if it has degree $d$ 
and if its multi-degree ${\ul d}=(\deg L|_{Z_1}, \dots, \deg L|_{Z_r})$ 
is ``not too far'' from 
$$\left(d\frac{\deg \omega_{|Z_1}}{\deg{\omega}}, 
\dots, d\frac{\deg \omega_{|Z_r}}{\deg{\omega}}\right),$$
where $Z_1,\dots,Z_r$ are the irreducible components of $C$
and $\omega$ is the canonical bundle of $C$.
In precise terms we impose
$$\left\lvert\deg L_{|Z}-d\frac{\deg \omega_{|Z}}{\deg{\omega_C}}\right\rvert<\frac{\#(Z\cap \ol{C\setminus Z})}{2}$$
for any proper subcurve $Z\subset C$.
Caporaso restricts the Picard functor of line bundles of total degree $d$ 
by imposing the above balancing condition: we get in this way 
 $\mathcal N_g^d\to \Mbar_g$.
It is a remarkable fact that the above inequality, originating  naturally in
Geometric Invariant Theory, under the hypotheses $(d-g+1,2g-2)=1, g\ge 3$, 
identifies a representable morphism  $\mathcal N_g^d\to \Mbar_g$  whose restriction on each  trait 
$\Spec R\to \Mbar_g$ transversal to the boundary  $\partial \Mbar_g=\Mbar_g\setminus \Mcal_g$
is the \ner  model of its generic fibre, the torsor $\Pic^d(C_K)$.

\begin{exa}
Consider a smooth curve $C_K$ 
of genus $3$ degenerating at the special point of 
$\Spec R$ to a curve $C_k$ with two smooth
components $C_1$ and $C_2$ of genus $1$ and two nodes. 
We can consider the above functor $\mathcal N_3^1$ which, over the special fibre, 
allows only multidegrees
$(0,1)$ and $(1,0)$. Indeed $\Pic^{(0,1)}C_k$ and 
$\Pic^{(1,0)}C_k$ are the two components of the special fibre of the N\'eron model of $\Pic^1C_K$.

In this example we can describe the scheme 
$\Pic^{(0,1)}C_k\sqcup \Pic^{(1,0)}C_k$ without ordering the two components with the indices
``1'' and ``2''; for this, we can use the expression ``balanced line bundles'', or 
we can unravel this notion in this case by saying that $\mathcal N_g^1\otimes R$ is the scheme parametrising line bundles of total degree-$1$
and nonnegative degree on all components. 
As we see in the next example, this can be done only because the hypotheses
$\gcd(d-g+1,2g-2)=1$ and $g\ge 3$ are satisfied.
\end{exa}

Over $\Mbar_{g,n}$ with $n\ge 1$ we 
can exploit the functor of balanced line bundle of total degree $d$ to form 
$\mathcal N_{g,n}^d$. Then, via 
the isomorphism $L\to L\otimes \Ocal(-d[x_1])$, where 
$x_1$ denotes the first marking, 
we provide a separable family whose restriction on every trait $\Spec R \to \Mbar_{g,n}$ 
transversal to the boundary 
locus $\partial \Mbar_{g,n}=\Mbar_{g,n}\setminus \Mcal_{g,n}$ 
is the N\'eron model of its generic fibre, the Jacobian $\Pic^{\zer}C_K$.

\begin{exa}
We examine again the previous example where the curve $C_K$, which we now assume  equipped with 
a distinguished marking $x_1$, 
 degenerates to the $2$-noded curve 
$C_1\cup C_2$, where we assume that $x_1$ lies on $C_1$. 
The isomorphism $\phi\colon L\mapsto L\otimes \Ocal(-[x_1])$ yields 
$\phi \mathcal N_3^1$ which, over the special fibre, consists of points paramtrizing 
line bundles of multidegrees
$(-1,1)$ and $(0,0)$. Indeed $\Pic^{(-1,1)}C$ and 
$\Pic^{(0,0)}C$ are the two components of the special fibre of the N\'eron model 
of $\Pic^{\zer}C_K$. 

Notice that this approach uses in a crucial way the marking $x_1$ 
which distinguishes one, privileged, irreducible component.
The scheme $\phi\mathcal N_{3,1}^1$ 
can never admit a group structure compatible with that of the Jacobian 
over $\Mcal_{g,n}$. 
\end{exa}

We now apply Corollary \ref{thm:global} and
 consider the functor $\Pical_g^{\zer,\ell}$, whose fibre over a point of 
$\partial \Mbar_{g}\upl=\Mbar_{g}\upl\setminus \Mcal_{g}$ 
 is a group of line bundles on an $\ell$-stable curve $\sta C$ 
 of degree zero on every 
 component. Furthermore this group is isomorphic to the special fibre 
 of N\'eron model of 
a smoothing of 
 the coarse space $C$ of $\sta C$.
Recall that $\ell$ should be a multiple of the exponent of any 
critical group of any genus-$3$ stable graph. 
In particular since the dual graph of $C_1\cup C_2$ is the $2$-cycle graph whose 
critical group is $\ZZ/2$, we should assume $\ell\in 2\ZZ$.

\begin{exa}
Consider again the smooth genus-$3$ curve $C_K$ 
degenerating at the special point of 
$\Spec R$ to a curve $C_k$ with two smooth genus-$1$
components $C_1$ and $C_2$. Over $R_\ell$ we consider the 
corresponding regular twisted model whose special fibre 
$\sta C$ has stabilisers of order $\ell$ at the two nodes 
$n$ and $n'$ locally described as 
$\{xy=0\}$ and $\{x'y'=0\}$.  
We can consider $\Pical_3^{\zer,\ell}$ which, over the special fibre, 
represents line bundles whose degree vanishes on each component
and such that $\gamma^*\sta L\cong \sta L,$ where $\gamma$ 
is the ghost automorphism $\gamma(x,y)=(\xi_\ell x, y)$
and $\gamma(x',y')=(\xi_\ell x', y')$.

We get a new description of the two components forming the special fibre of the 
N\'eron model of the Jacobian of $C_K$. 
One component parametrises line bundles on $\sta C$
arising as pull-backs from the coarse space $C$; namely, locally at the node, 
the $\pmmu_\ell$-action is $$\frac1\ell(1,\ell-1,0)$$ where 
the entries $1$ and $\ell-1$ describe the action on the curve and the third entry 
describes the trivial action on the fibre of the line bundle $\sta L$.
The second component parametrizes line bundles which still have degree $0$ on both 
components ad whose $\pmmu_\ell$-action on each node is given by 
$$\frac1\ell(1,\ell-1,\frac\ell2).$$ 
As observed in Remark \ref{rem:otherell}, when we are only interested in a local 
picture over $R$, we can choose $\ell=2$.
The advantage of the approach via $\Pical_{g}^\ell$ is the
that we do not need to privilege one component among $C_1$ and $C_2$. 
\end{exa}

We can now place the attention on a  universal 
deformation of $C_1\cup C_2$. 
The deformation space $U$ contains a normal 
crossing divisor with two irreducible components
$D'=(u=0)$ and $D''=(v=0)$ 
parametrising the deformation along which $n_2$ is smoothed and $n_1$ persists 
and the deformation along which $n_1$ is smoothed and $n_2$ persists. 
Caporaso's functor of balanced Picard groups allows us to patch together 
the \ner models along each trait $u=\la v$. 

We can also apply Theorem \ref{thm:local2}. The scheme $\phi 
\mathcal N_3^1$ cannot be equipped with a group structure. 
On the other hand, when we consider $\wt{U}$ the spectrum of the ring 
obtained by extracting a $\ell$th roots $\wt u$ and $\wt v$
from the parameters $u$ and $v$, 
we can consider a map $\wt{U}$ to $\Mbar_g\upl$ and 
a family $\sta C$ of $\ell$th twisted curves over $\wt{U}$ whose coarse space descends 
to a family $C$ of stable curves over $U$. There we can 
consider the functor $\Pic^{\zer,\ell}$ and its restriction  
$\Pic^{\zer,\ell}_{\wt U}$ on the scheme $\wt U$. 
We get a group scheme  whose restriction 
on each trait $\Spec \wt R\to \wt{U}$ mapping $\pi\mapsto (\pi,\la \pi)$
with $\la \neq 0$ is the pullback of the \ner model of the Jacobian of 
the smooth curve $\Spec( K^{\pmmu_\ell}) \times_U C$. 

\bigskip 

One can now address a more ambitious question: namely, can we provide a group scheme 
over a universal deformation of $C_1\cup C_2$
containing the \ner models of any Jacobian $\Jac C_K$ for any injective 
trait $\Spec R \to U$ 
whose generic point maps into the open locus of $U\setminus (D\cup D')$?
It is not difficult to see that such a group scheme cannot be of finite type; 
indeed we can approach the special point 
$u=v=0$ by mapping $\pi$ to $(\pi^n, \la \pi^m)$ with $\la\neq 0$ $\gcd(n,m)=1$ and 
the special fibre of the \ner model of the 
Jacobian of the generic curve is a group scheme with $n+m$ components. 
The work of Holmes \cite{Holmes} provides a positive answer by modifying $U$ as follows. 
\begin{exa}\label{exa:Holmes}
Let $U_0$ be $U$, let $D_0$ be $D\cup D'$, and let $B_0$ be $\Sing(D_0)=\{u=0,v=0\}$. Then,
we set $U_1=\mathrm{Bl}_{B_0}(U_0)$ mapping via $f_1$ to $U_0$, 
$D_1=f^*D_0$, and $B_1=\Sing D_1$. We iterate by  blowing up $B_1$ within $U_1$, \emph{etc}.
We now consider the co-limit $V$ induced by $V_0=U_0\setminus B_0\hookrightarrow 
V_1=U_1\setminus B_1\hookrightarrow V_2=U_2\setminus B_2\hookrightarrow \dots$
with a divisor induced by the co-limit $D_0\hookrightarrow D_1\hookrightarrow D_2\hookrightarrow
\dots$
Any trait $\Spec R\hookrightarrow 
U$ mapping the generic point to $U_0\setminus D_0$ naturally lifts to $V$ and 
Holmes provides a group scheme $N_V$ over $V$ such that the 
\ner model of $\Jac(C_U\otimes K)=N_V\otimes K$ is $N_V\otimes R$.
\end{exa}

This construction is all the more satisfactory because 
it arises as the special fibre over the deformation space $U$ of 
a global 
group scheme $\wt {\mathcal N}_g$, which is a model of the 
the universal 
Jacobian $\Pic^{\zer}$ over $\Mcal_g$ and satisfy an analogue of the
\ner property over a birational non-proper 
modification $\wt \Mcal_g\to \Mbar_g$ restricting to an isomorphism on $\Mcal_g$.
Indeed in \cite{Holmes}, Holmes produces $\wt{\mathcal N}_g$ as the terminal 
object in a category of ``\ner admitting morphisms'', see Defn.~1.1 and 1.2 in \cite{Holmes}. 

In the above construction both $\wt {\mathcal N}_g$ and the base $\wt {\mathcal M}_g$ 
are not of finite type. We point out that Olsson's stack $\mathfrak M_g^\ttt$ 
of all twisted curves (see \cite{Olsson}) can be also equipped with a group scheme $\mathfrak{Pic}_g^{\zer,\ttt}$ 
parametrising line bundles of degree zero on all components.
It should be noticed that $\mathfrak M_g^\ttt$ is not of finite type nor separated
(the stack of $\ell$-stable curves was introduced precisely in order to identify a 
proper substack of $\mathfrak M_g^\ttt$). 
 Reformulating 
\ref{thm:local2}, we can 
obtain \ner models for any trait as follows. 
%

\begin{cor}\label{cor:anothertake}
 Consider a stable curve $C_R$ over a discrete valuation ring $R$ 
 whose fibre over $K$ is the smooth curve $C_K$ and whose 
 special fibre $C_k$ has thicknesses $\bf t$,
 dual graph $\Gamma_k$ 
 and $\bf t$-critical group $\mathcal K_{\bf t}(\Gamma_k)$.
 
For any positive integer  $\ell$, we 
consider the regular twisted model $\sta C^\ttt$ over $R$ 
and the corresponding morphism $\Spec R_\ell \to \mathfrak M^\ttt_g$.  
The $\Pic^{\zer}$ functor is a group scheme over $\mathfrak M^{\ttt}_g$ and 
the pullback over $R_\ell$ is a $\pmmu_\ell$-equivariant morphism $\Pic^{\zer}_{R_\ell}
\to \Spec R_\ell$ whose special fibre splits into the disjoint union of a 
$\pmmu_\ell$-fixed
part $F_\ell$ and a $\pmmu_\ell$-moving
part   $M_\ell$.
 
Then, as soon as $\ell$ is a multiple of the exponent of the $\bf t$-critical group 
$\mathcal K_{\bf t}(\Gamma_k)$, 
the \ner model of the Jacobian $\Pic_K$
of the generic fibre coincides, after pullback to $R_\ell$, with the complement 
of the moving part $M_\ell$ of the 
base change of $\mathfrak{Pic}^{\zer,\ttt}_g\to \mathfrak{M}_g^\ttt$ via 
$\Spec R_\ell\to \mathfrak M_{g}^\ttt$.
Summarising, we have the following fibre diagrams
$$
\xymatrix @!0 @R=1cm @C=1.7cm {
 \Pic_K\ar[rr]\ar[dd]    && 
           N(\Pic_K)\ar[dd]  && 
                \Pic^{\zer,\ell}_{R_\ell}=\Pic^{\zer}_{R_\ell}\setminus M_\ell
                \ar[dd] \ar[ll]\ar[rr]^{\ \ \ \ \ \ \ \subseteq} && 
               \Pic^{\zer}_{R_\ell}\ar[dd] \ar[rr] &&
                       \mathfrak{Pic}^{\zer,\ttt}_{g}\ar[dd] 
        \\
                             &\square&    &\square&    && &\square&  \\
        \Spec K \ar[rr]  && 
           \Spec R && 
              \Spec R_\ell \ar[ll]&& \Spec R_\ell 
              \ar[rr]&&
                   \mathfrak M_g^{\ttt}.
        }
$$\qed
\end{cor}

\begin{rem}
The statement above does not hold as soon as the index $\ell$  is not a multiple of 
the exponent of $\mathcal K_{\bf t}(\Gamma_k)$. 
Furthermore we point out that \eqref{eq:ordcrit} provides an explicit formula for an index 
$\ell$ for which the statement holds --- a 
multiple of the exponent of the $\bf t$-critical group. 
\end{rem}
%
%
%

\vspace{.5cm}

{\tiny
\noindent{Institut de Math\'ematiques de Jussieu, 
UMR 7586, CNRS, UPMC, case 247, 4 Place Jussieu, 75252 Paris cedex 5, France\\
\textit{E-mail address:} \url{Alessandro.Chiodo@imj-prg.fr}}
}
\vspace{.3cm}
\end{document}